\documentclass[11pt,a4paper,oneside]{article}
\usepackage[top=3cm, bottom=3cm, left=2cm, right=2cm]{geometry} 
\usepackage[english]{babel}
\usepackage[utf8]{inputenc}
\usepackage[T1]{fontenc}

\usepackage{amsthm,amsmath,amssymb,mathrsfs,dsfont,mathtools}
\usepackage{booktabs, subcaption}
\usepackage{bbm}
\usepackage{multirow}
\usepackage{bm}
\usepackage{graphicx}
\usepackage[colorlinks=true, allcolors=blue]{hyperref}
\usepackage[noabbrev, capitalise]{cleveref}
\usepackage{comment}
\usepackage{microtype} 
\usepackage[all]{nowidow} 

\usepackage{caption}
\usepackage{xparse}
\usepackage{mathtools}

\usepackage[authoryear, round]{natbib}
\newcommand{\e}{{\rm e}}

\newcommand{\ind}{\stackrel{\mbox{\scriptsize ind}}{\sim}}
\newcommand{\deq}{\stackrel{\mbox{\scriptsize d}}{=}}
\newcommand{\indicator}{\ensuremath{{1}}}

\newcommand{\calC}{\mathcal{C}}
\newcommand{\calT}{\mathcal{T}}
\newcommand{\calM}{\mathcal{M}}

\newcommand{\calP}{\mathcal{P}}

\newcommand{\dd}{\mathrm d}

\newcommand{\plaw}{\mathsf{P}}
\newcommand{\partder}[1]{\frac{\partial}{ \partial #1}}
\newcommand{\E}{E}
\newcommand{\Xcr}{\mathscr{X}}
\newcommand{\Ycr}{\mathscr{Y}}
\newcommand{\Acr}{\mathscr{A}}

\NewDocumentCommand{\evalat}{sO{\big}mm}{%
  \IfBooleanTF{#1}
   {\mleft. #3 \mright|_{#4}}
   {#3#2|_{#4}}%
}

\newcommand{\Pp}{\mathbb{P}}
\newcommand{\X}{\mathbb{X}}

\newcommand{\M}{\mathbb{M}}
\newcommand{\Y}{\mathbb{Y}}

\newcommand{\R}{\mathbb{R}}

\newcommand{\prob}{{\rm pr}}

\renewcommand{\mid}{\ensuremath{\,|\,}}

\def\simind{\stackrel{\mbox{\scriptsize{\rm ind}}}{\sim}}

\usepackage{authblk}
\date{}
\title{Palm distributions of superposed point processes \\for statistical inference}
\author[1]{Mario Beraha}
\author[1]{Federico Camerlenghi}
\author[2]{Lorenzo Ghilotti}

\affil[1]{Department of Economics, Management and Statistics, University of Milano-Bicocca, Milan, Italy}
\affil[2]{Department of Statistical Science, Duke University, Durham, NC, USA}

\providecommand{\keywords}[1]{
  \small 
  \textbf{\textit{Keywords:}} #1
  \normalsize
}

\usepackage[sf]{titlesec}
\usepackage{sectsty}
\allsectionsfont{\centering \normalfont\scshape} 

\newtheorem{theorem}{Theorem}
\newtheorem{corollary}{Corollary}
\newtheorem{lemma}{Lemma}
\newtheorem{proposition}{Proposition}
\theoremstyle{definition}

\theoremstyle{remark}
\newtheorem{example}{Example}

\begin{document}

\maketitle

\begin{abstract}
Palm distributions play a central role in the study of point processes and their associated summary statistics. In this paper, we characterize the Palm distributions of the superposition of independent point processes, establishing a simple mixture representation depending on the point processes' Palm distributions and moment measures.
We explore two statistical applications enabled by our main result. First, we consider minimum contrast estimation for corrupted point processes. Second, we investigate the class of shot noise Cox processes and derive explicit expressions for their higher-order Palm distributions. In the finite case, we further obtain a tractable expression for the Janossy density, which plays the role of a likelihood function and thus can be used for new likelihood-based inference strategies.
Extensions to the superposition of multiple point processes and to higher-order Palm distributions are also presented.
\end{abstract}


\section{Introduction} \label{sec1}

Real-world point patterns often combine several structured components, some regular or inhomogeneous, others explicitly clustered, and may also include random noise. 
Semiconductor wafer defect maps \citep{Borgons2021}, disease-case locations in epidemiology \citep{meyer2017surveillance},  cellular network base-station layouts \citep{choi2017gridppp}, mixed age tree stands in ecology \citep{ngobieng2011forest}, and earthquake aftershock sequences \citep{ogata1998space} can all be regarded as superpositions of two (or more) independent point processes.
While the superposition operation is trivial at the model level, from an inferential standpoint, superposed point processes are notoriously awkward. 
Standard tools such as minimum contrast estimation rely on closed-form expressions for second-order summaries, such as the $J$-function, Ripley's $K$-function, and Besag's $L$-function. However, those summaries remain unknown for a generic superposition \citep{MoWaBook03,diggle2013statistical}. 
Hence, practitioners need to rely on complex algorithms often developed on a case-by-case basis \citep[e.g.,][]{tanaka2014superposed,xu2018hawkes}.

In this work we first focus on the Palm distributions of the superposition of two independent point processes. Palm distributions \citep{BaBlaKa} are key mathematical objects in the study of point processes, describing the conditional behavior of a process given the location of one or more of its points or \emph{atoms}.
We establish a mixture representation for the Palm distributions of the superposed process, in which the mixture components are sums of the two original processes and their Palm versions, weighted by their respective mean measures.
We also show that our analysis easily extends to the superposition of more than two independent processes.

We demonstrate the practical usefulness of our result for statistical inference in two settings.
First, we consider fitting a point process contaminated by random background noise via minimum contrast estimation.
Indeed, through the Palm distributions, it is straightforward to obtain functional summary statistics, enabling robust and fast inference via minimum contrast. 
This methodology is particularly relevant in applied contexts such as the analysis of spatial defect structures in semiconductor manufacturing \citep{Borgons2021}.

Second, we investigate structural and distributional properties of the shot noise Cox process \citep[\textsc{sncp},][]{Mo03Cox}, a prominent class of cluster processes. Cluster processes are routinely employed in several applied areas, including astronomy, materials science, and plant ecology; see, e.g., \citet{Illian2008,MoWaBook03}. Despite their flexibility and wide applicability, several distributional aspects of \textsc{sncp} models remain only partially understood.
As an application of our main result, we derive the higher-order Palm distributions of the \textsc{sncp}, which were previously unavailable in the literature. In the case of finite \textsc{sncp}s, we further obtain an explicit expression for the Janossy density. Since the Janossy density plays the role of a likelihood function for finite point processes, this result paves the way for new likelihood-based inference strategies in the context of \textsc{sncp} models.
We also discuss the use of our results in other statistical contexts, both frequentist and Bayesian.



\section{Superposition of point processes}

\subsection{Background and notation for point processes} \label{sec:notation}

Let $\X$ be a Polish space equipped with the corresponding Borel $\sigma$-algebra $\Xcr$.  
A point process $\Phi$ on $\X$ can be represented as $\Phi = \sum_{j \geq 1} \delta_{X_j}$, where $(X_j)_{j\geq 1}$ is a sequence of random variables (\textit{atoms}) taking values in $\X$, and  $\delta_{x}$ denotes the Dirac delta mass at $x$. The probability distribution of $\Phi$ is denoted with $\plaw_{\Phi}$. The number of atoms in $\Phi$ could be either finite or infinite.
In the sequel, we will follow the approach of \cite{BaBlaKa}, where $\Phi$ is regarded as a random counting measure. See Section \ref{app:point_process_definition} for further mathematical details, including the $\sigma$-algebra on the space of counting measures. 

Let $M_{\Phi}(B) = \E[\Phi(B)]$, for $B \in \Xcr$, be the mean measure of $\Phi$, and define the $k$-th factorial moment measure 
$M_{\Phi}^{(k)}$ as the mean measure of the $k$-th factorial power of $\Phi$,  i.e., of the point process $\Phi^{(k)}$ defined as:
\[
    \Phi^{(k)} := \sum_{(j_1, \ldots, j_k)}^{\not =}  \delta_{(X_{j_1}, \ldots, X_{j_k})},
\]
where the symbol $\not =$ means that the sum is extended over all pairwise distinct indexes.

To introduce the notion of \emph{Palm distributions}, let us define the Campbell measure of $\Phi$, namely $\mathscr C_{\Phi}(B \times L) := \E[\Phi(B) \indicator(\Phi \in L)]$, where  $B \in \Xcr$ and  $L$ is an element of the appropriate $\sigma$-algebra on the space of random counting measures (see Section \ref{app:point_process_definition}).
Under the assumption that $M_\Phi$ is $\sigma$-finite, it can be shown that $\mathscr C_{\Phi}$ admits the following representation
\[
   \mathscr C_{\Phi}(B \times L) = \int_B \plaw_{\Phi}^x(L) M_\Phi(\dd x),
\]
where $\{\plaw_{\Phi}^x\}_{x \in \X}$ is the almost surely (a.s.) unique disintegration probability kernel of $\mathscr C_\Phi$ with respect to $M_\Phi$, and it is referred to as the Palm kernel of $\Phi$.
For fixed $x \in \X$, $\plaw_{\Phi}^x$ is a probability distribution over the space of counting measures, termed the Palm distribution of $\Phi$ at $x$, and thus it can be identified with the law of a point process $\Phi_x \sim \plaw_{\Phi}^x$, which is consequently called a \emph{Palm version} of $\Phi$ at $x$. By Proposition 3.1.12 in \cite{BaBlaKa},  for $M_\Phi$-almost all $x \in \X$, the point process $\Phi_x$ contains the atom $x$ with probability $1$. This justifies the interpretation of the Palm distribution of $\Phi$ at $x$ as the law of $\Phi$ conditionally to $\Phi$ having an atom at $x$. In addition, the point process $\Phi^!_x := \Phi_x - \delta_x$ is well-defined, and $\Phi^!_x$ is called a \emph{reduced Palm version} of $\Phi$ at~$x$. 
Finally, it is possible to extend the definition of Palm distributions to multiple conditioning points $\underline{x} = (x_1, \ldots, x_k) \in \X^k$. In this case, the Palm distribution of $\Phi$ at $\underline x$ is interpreted as the probability distribution of $\Phi$ conditionally to $\Phi$ having $k$ atoms at locations $x_1, \ldots, x_k$. See Section \ref{app:point_process_definition}.

\subsection{Palm distributions of the superposition of independent processes} \label{sec:main}

From now on, consider independent simple point processes $\Phi_i$, $i = 1, \ldots, m$. Their superposition $\Phi$ is defined as
$\Phi := \sum_{i= 1}^m  \Phi_i$, i.e., the union of all the point patterns. 
The following theorem states the main theoretical result of the paper, which characterizes the Palm distributions of the superposition of two independent point processes.  

\begin{theorem}\label{teo:main}
    Let $\Phi_1$ and $\Phi_2$ be two independent point processes on $\X$. Then, for any $x \in \X$, the Palm version $(\Phi_1 + \Phi_2)_{x}$ can be expressed as the following mixture:
    \[
        (\Phi_1 + \Phi_2)_{x} \deq \begin{cases}
            \Phi_{1 x} + \Phi_2 & \text{with probability equal to }  \frac{\dd M_{\Phi_1}}{\dd M_\Phi}(x) \\
            \Phi_{1} + \Phi_{2 x} & \text{with probability equal to } \frac{\dd M_{\Phi_2}}{\dd M_\Phi}(x).
        \end{cases}
    \]
\end{theorem}
The proof is reported in the Appendix at the end of the paper. Theorem \ref{teo:main} admits an intuitive interpretation about the Palm distributions of the superposed process. Conditioning the superposed process $\Phi= \Phi_1+\Phi_2$ on having a point at $x$, its distribution depends on whether $x$ originates from $\Phi_1$ or from $\Phi_2$, thus reflecting the respective contributions of each source.
The mixture structure in the Palm version $(\Phi_1 + \Phi_2)_{x}$  explicitly accounts for the two mutually exclusive scenarios. 
The mixing probabilities correspond precisely to the probabilities that point $x$ originates from one or the other process.
Furthermore, exploiting the relation $\Phi_{i x} \deq \Phi_{i x}^! + \delta_x$, for $M_{\Phi_i}$-almost all $x$, $i=1, 2$, we can replace all the Palm versions appearing in the statement of \Cref{teo:main} with their reduced counterparts.

To generalize Theorem 1 to $m$ processes and $k$ points, let us introduce latent allocation variables $T_1,\ldots,T_k\in\{1,\ldots,m\}$ with the interpretation that $T_j=i$ iff $x_j$ was generated from $\Phi_i$. 
\begin{theorem}\label{teo:multiple}
Let $\Phi=\sum_{i=1}^m \Phi_i$ be the superposition of the independent simple point processes $\Phi_1,\ldots,\Phi_m$, and $\underline{x}=(x_1,\ldots,x_k)\in \X^k$, with pairwise distinct coordinates. Assume that the $k$-th factorial moment measure $M_{\Phi}^{(k)}$ is $\sigma$-finite, and that $M_{\Phi_i}^{(r)}$ admits a density $\rho_{\Phi_i}^{(r)}$ w.r.t. $\mu^{\otimes r}$, namely the $r$-fold product of a reference measure $\mu$, for all $i=1, \ldots , m$ and $r\geq 1$.
For $\underline{T} := (T_1,\ldots,T_k)$, define $\underline{x}_i = ( x_j: T_j = i )$ and denote by $n_i$  the cardinality of $\underline{x}_i$.
Then, for $M_{\Phi}^{(k)}$-almost all $\underline{x}$,
\[
    \Phi^!_{\underline{x}}\mid (\underline{T}=\underline{t})\ \deq\ \sum_{i=1}^m (\Phi_i)^!_{\underline{x}_{i} }, \qquad P(\underline{T}=\underline{t})\propto \prod_{i=1}^m \rho_{\Phi_i}^{(n_i)} \big(\underline {x}_i \big).
\]
where, conditionally on $\underline{T}=\underline{t}$, the $(\Phi_i)^!_{\underline{x}_{i} }$'s are mutually independent and, when $n_i=0$,  $(\Phi_i)^!_{\underline{x}_{i} } = \Phi_i$ and $\rho_{\Phi_i}^{(n_i)} \equiv 1$.
\end{theorem}


\section{Inference for corrupted processes via minimum contrast} \label{sec:corruzione}

\subsection{Summary statistics for superposed point processes} \label{sec:summary_stat}

An application of Theorem \ref{teo:main} yields closed-form and interpretable expressions for commonly used summary statistics.
Here we concentrate on two summary statistics for stationary point processes, several additional examples are provided in the supplementary materials. Throughout this section, we assume that $\Phi = \Phi_1 + \Phi_2$ where the $\Phi_j$'s are stationary with intensity  $ \rho_j$, and let $\rho = \rho_1 + \rho_2$.

Perhaps the most commonly used summary statistic is Ripley's $K$-function \citep{ripley1976second}. Thanks to Theorem \ref{teo:main}, for a superposed point process $\Phi$ this takes the following form, for $r \ge 0$,
\begin{equation} \label{eq:K}
    K_{\Phi}(r) := \frac{1}{\rho} \E[\Phi^!_{o}(B(o,r))] =
   \frac{1}{\rho} \left[ K_{\Phi_1}(r) \frac{\rho_1^2}{\rho}+ K_{\Phi_2}(r) \frac{\rho_2^2}{\rho}+ 2 |B(o,r)| \frac{\rho_1\rho_2}{\rho}   \right],
\end{equation} 
where $o$ denotes a generic point of $\X$, and $|B(o, r)|$ is the volume of the ball $B(o, r)$ with radius $r$ centered at $o$, given that  $ B(o, r) \subset \X$.
Thanks to stationarity, $K_{\Phi}$ is invariant to the choice of point $o$, which is called the \emph{typical point} of $\Phi$: the quantity $\rho K_{\Phi}(r)$ represents the expected number of points that are $r$-close to a generic point $o$, given that $\Phi$ has an atom at $o$.  Note that, as pointed out by an anonymous reviewer, expression $\eqref{eq:K}$ can also be obtained by ad hoc computations, without applying Theorem \ref{teo:main}. 
A second example of summary statistic  is the reduced Palm distribution generating function $A$, defined by \cite{chiu2008reduced}, whose expression equals
\begin{equation} \label{eq:A}
    A_{\Phi}(s, r) := \E\left[s^{\Phi^!_o(B(o, r)}\right] = \frac{\rho_1}{\rho} A_{\Phi_1}(s, r) \E\left[s^{\Phi_2(B(o, r)}\right] + \frac{\rho_2}{\rho} A_{\Phi_2}(s, r) \E\left[s^{\Phi_1(B(o, r)}\right], 
\end{equation}
by virtue of Theorem \ref{teo:main}. Contrary to the $K$-function, the $A$-function captures higher-order characteristics of the point process and, as shown in \cite{chiu2008reduced}, is particularly suited for cluster point processes with regularity. Finally, in Section \ref{app:stat_summary} we also provide an explicit expression for the nearest-neighbor distance distribution function in the stationary case, while in Section \ref{app:nonstat_summary} we exploit Theorem \ref{teo:main} to evaluate summary statistics for non-stationary point processes.

\subsection{Fitting a corrupted Mat\'ern cluster process via minimum contrast estimation} \label{sec:corrupted_Matern}

Minimum contrast estimation methods \citep[\textsc{mce},][]{MoWaBook03} are a class of techniques for fitting parametric point process models to observed point patterns.
In a nutshell, given a functional summary statistic, such as $K_{\Phi}(r)$, and its nonparametric estimator $\hat K(r)$, the parameters of $\Phi$ are chosen to minimize a suitable distance between the theoretical functional summary statistic and its realized estimator, integrated over a specified interval for $r$. Typically, the distance is taken to be an $L_q$ distance between functions; see \cite{diggle2013statistical} for further details. Similarly, one can use other summary statistics, such as the $A$-function instead of the more common Ripley's $K$-function.
Here,  we describe an application of \textsc{mce} in which $\Phi_1$ is a Mat\'ern cluster process, corrupted by a background noise $\Phi_2$. An additional example is presented in Section \ref{app:corrupted_dpp}, where we fit a corrupted repulsive point process.


For the sake of illustration, let $\X= \R^2$. Define the kernel function $\kappa ( \xi ; c ) = \indicator_{B (c, R)} (\xi)/ (\pi R^2)$,
where $R >0$. A Mat\'ern cluster process $\Phi_1$ is a Cox process defined in a  hierarchical fashion:
$\Phi_1 \mid \Phi_p$ is a Poisson process with intensity measure $\int_{\R^2}  \mu  \kappa (\xi ; c) \Phi_p (\dd c) \dd \xi$, and $\Phi_p \sim \textsc{pp} (\lambda \dd c)$, 
where $\textsc{pp} (\lambda \dd c)$ denotes a  Poisson process with intensity $\lambda$ on $\X$, and $\lambda, \mu >0$.
Equivalently, $\Phi_1$ is a cluster process: the parent points are generated by the latent  parent process $\Phi_p$, while each cluster of $\Phi_1$ consists of points uniformly distributed in $B (c_j, R)$, where $\Phi_p= \sum_{j \geq 1} \delta_{c_j}$.


We assume to observe a realization of the Mat\'ern cluster process $\Phi_1$ corrupted by a background noise, independent of $\Phi_1$. We further suppose that the corrupting noise  $\Phi_2$ is a homogeneous Poisson point process with intensity $\rho_2$.
We propose to fit the superposition $\Phi_1 + \Phi_2$ to the observed point pattern via \textsc{mce}. 
We compare three different approaches: two based on the $A$- and $K$-functions of the superposed process, and one where we ignore the superposition and fit \textsc{mce} based on the $A$-function of $\Phi_1$ (this is equal to setting $\rho_2 \equiv 0$). Refer to Section \ref{app:details_matern} for the explicit expressions of $A$ and $K$ for our Mat\'ern--Poisson superposition.
We generate data from a Mat\'ern point process with parameters $\rho_1 = 50$, $\mu=5$, $R = 0.05$ and corrupt it with a Poisson process with $\rho_2 = 20$. We consider two observation windows, namely $[0, 1]^2$ and $ [0, 5]^2$.

\begin{table}[t]
	\centering
\begin{tabular}{llcccc}
Window & Method  & $\rho_1$ & $\mu$ & $R$ & $\rho_2$ \\
\midrule
\multirow{3}{*}{$[0,1]^2$}
& $A$-\textsc{mce} (correct)
& 53.7 (24.5) & 5.42 (2.05) & 0.05 (0.01) & 19.3 (30.8) \\
& $K$-\textsc{mce} (correct)
& 36.1 (16.4) & 7.09 (3.65) & 0.05 (0.02) & 57.5 (16.8) \\
& $A$-\textsc{mce} (misspecified)
& 67.5 (29.9) & 4.60 (1.37) & 0.05 (0.01) & --- \\
\midrule
\multirow{3}{*}{$[0,5]^2$}
& $A$-\textsc{mce} (correct)
& 50.1 (4.9) & 5.03 (0.39) & 0.05 (0.00) & 20.2 (6.7) \\
& $K$-\textsc{mce} (correct)
& 42.1 (6.3) & 5.54 (0.70) & 0.05 (0.00) & 36.8 (4.9) \\
& $A$-\textsc{mce} (misspecified)
& 64.2 (3.7) & 4.28 (0.24) & 0.05 (0.00) & --- \\
\end{tabular}
\caption{Median and interquantile range (\textsc{iqr}, in brackets) of the estimates for the superposed Mat\'ern--Poisson model over $200$ independent replicated
datasets.}
\label{tab:mce_A_K_two_windows}
\end{table}

The results reported in Table \ref{tab:mce_A_K_two_windows} highlight clear differences among the three estimators. The $A$-\textsc{mce} under the correct superposed model is well centered around the true parameter values, and its dispersion decreases significantly as the observation window grows, as expected from increasing information. In contrast, the $A$-\textsc{mce} that ignores the Poisson background exhibits substantial bias in estimating $\rho_1$, while still delivering accurate estimates of $\mu$ and $R$. Finally, the $K$-\textsc{mce} tends to overestimate $\rho_2$ and consequently underestimate the $\rho_1$, underscoring that capturing features beyond second-order characteristics can be crucial for reliable inference in cluster processes, especially under superposition.

\section{Shot noise Cox processes: Palm distributions and Janossy densities} \label{sec:sncp_inference}

\textsc{sncp}s are a class of general models for clustered point patterns.
To define a \textsc{sncp},  let $\nu(\dd \theta \, \dd \gamma)$ be a locally finite diffuse intensity measure on $\X \times\R_+$. Let $\{\kappa(\cdot; \theta)\}_{\theta \in \X}$ be a family of parametric probability density functions on $\X$, called the kernel of the \textsc{sncp}, where the parameter space is $\X$ itself.
Assuming that $\int_{\X\times \R_+} \gamma \kappa(x; \theta) \nu(\dd \theta \, \dd \gamma) < \infty$ for any $x \in \X$, then $\Phi$ is a \textsc{sncp} directed by $\nu$ with kernel $\kappa$ if 
\begin{equation}\label{eq:sncp_definition}
    \Phi \mid \Lambda \sim \textsc{pp}\left(\int_{\X\times \R_+} \gamma \kappa(x; \theta) \Lambda(\dd \theta \, \dd \gamma) \dd x\right), \qquad \Lambda \sim \textsc{pp}(\nu).
\end{equation}
We write $\Phi \sim \textsc{sncp}(\kappa, \nu)$.
A \textsc{sncp} is a \emph{cluster process}: the coloring theorem of Poisson processes entails that $\Phi$ equals in distribution the sum $\sum_{i \geq 1} \Phi_i$, where $\Phi_i \mid \Lambda \ind \textsc{pp}(\gamma_i \kappa(x; \theta_i) \dd x)$, and $\Lambda = \sum_{i \geq 1} \delta_{(\theta_i, \gamma_i)}$. 
Then, for each point $X_j$ of $\Phi$, it is possible to introduce a latent cluster allocation variable $T_j$ such that $\Phi_i(\{X_j\}) = 1$ iff $T_j = i$ and $\Phi_i(\{X_j\}) = 0$ otherwise, i.e., $T_j = i$ if $X_j$ arose from $\Phi_i$. 
The number of distinct values across $\underline{T} := (T_j)_{j\geq 1}$, denoted with $|\underline{T}|$, represents the number of clusters among the points. 
\cite{wang2023spatiotemporal} exploit this construction to draw a connection with Bayesian mixtures of finite mixtures \citep{Lijoi08,miller2018mixture} when $\Lambda$ is a finite Poisson process, and the $\gamma_i$'s are independent and gamma-distributed.

The Palm distributions of a \textsc{sncp} at a single point $x$ was obtain in \cite{Mo03Cox}, but higher-order Palm and reduced Palm distributions for \textsc{sncp}s are not known in the literature. We fill this gap with the theorem below, which follows from a recursive application of \Cref{teo:main}. 
Define $\eta(x_1, \ldots, x_p) = \int_{\X\times \R_+} \gamma^{p} \prod_{j=1}^p \kappa(x_j; \theta) \nu(\dd \theta\, \dd \gamma)$ and assume that $\eta(x_1, \ldots, x_p) < \infty$, for any $x_1, \ldots, x_p$. The latter assumption imposes constraints on the choice of $\kappa$ and $   \nu$.  

\begin{theorem}\label{thm:red_palm_sncp} 
Let $\Phi \sim \textsc{sncp}(\kappa, \nu)$ and $\underline{x} = (x_1, \ldots, x_k) \in \X^k$. Then, the reduced Palm version of $\Phi$ at $\underline{x}$ admits the following representation
\begin{equation}\label{eq:thesis_palm_sncp}
    \Phi^!_{\underline{x}}\mid \underline{T} \deq \Phi + \sum_{h = 1}^{|\underline{T}|} \Phi_{\zeta_{{\underline{x}}_h}},\qquad \prob(\underline{T} = \underline{t}) \propto \prod_{h=1}^{|\underline{t}|} \eta(\underline{x}_h),
\end{equation}
where $\underline{T} := (T_1,\ldots,T_k)$ are latent indicators describing a partition of $\underline{x}$ into $|\underline{T}|$ clusters and
\begin{equation*}
    \begin{aligned}
        \Phi_{\zeta_{{\underline{x}}_h}} \mid \zeta_{{\underline{x}}_h} = (\theta_{{\underline{x}}_h}, \gamma_{{\underline{x}}_h}) &\sim \textsc{pp}(\gamma_{{\underline{x}}_h} \kappa(x; \theta_{{\underline{x}}_h}) \,\dd x)\\
        \zeta_{{\underline{x}}_h}  &\sim f_{{\underline{x}}_h}(\dd \theta\, \dd \gamma) \propto \gamma^{n_h} \prod_{x_j \in 
        \underline{x}_h}\kappa(x_j;\theta) \nu(\dd \theta\, \dd \gamma),
    \end{aligned}
\end{equation*}
where $\underline{x}_h = (x_j: T_j = h)$ and $n_h$ is the cardinality of $\underline{x}_h$.
Finally, the processes $\Phi$ and $\Phi_{\zeta_{{\underline{x}}_h}}$, $h = 1,\ldots, |\underline{T}|$, are mutually independent conditionally to $\underline{T}$.
\end{theorem}
See Section \ref{app:auxiliary_results} for the precise definition of the space where $\underline{T}$ takes value. \\

Theorem \ref{thm:red_palm_sncp} plays a key role in deriving the likelihood of a \textsc{sncp}, thereby paving the way for new strategies in maximum likelihood estimation. Assume that $\Phi$ has almost surely a finite number of points.
Following \cite{DaVeJo1}, we define the likelihood of $\Phi$ as its Janossy density seen as a function of the parameters of $\Phi$.
Briefly, we recall that for a regular finite point process $\Phi$ with associated Janossy density $j_k: \X^k \to \R_+$, $j_k(x_1, \ldots, x_k) \dd x_1 \,\ldots\, \dd x_k$ represents the probability that $\Phi$ consists of exactly $k$ points located in infinitesimal neighborhoods of $x_j$, $j=1, \ldots, k$. Under suitably regularity conditions, the family of Janossy densities $j_k(\cdot)$, $k\geq 0$, characterizes its probability distribution \citep[Proposition 5.3.II,][]{DaVeJo1}.
The next theorem, derived from \Cref{thm:red_palm_sncp}, gives an explicit expression for the Janossy densities when $\Phi$ is a general finite \textsc{sncp}. We remark that a sufficient condition for the finiteness of $\Phi$ is $\int_{\X\times \R_+} \gamma \nu(\dd \theta\,\dd \gamma) < \infty$. For simplicity, we report here only the case when $\nu(\dd \theta \, \dd \gamma) = \rho(\dd \gamma) G_0(\dd \theta)$; see the supplementary materials (Section \ref{app:general_janossy_sncp}) for the general statement and the proof.
\begin{theorem}\label{thm:janossy_sncp}
    Let $\Phi \sim \textsc{sncp}(\kappa, \nu)$ such that $\Phi(\X) < \infty$ almost surely and $\nu(\dd \theta \, \dd \gamma) = \rho(\dd \gamma) G_0(\dd \theta)$. Then 
\begin{equation}\label{eq:janossy_fact}
    j_k(x_1,\ldots,x_k) = k! \, \prob( \Phi(\X) = k) \,\E\left[\prod_{h=1}^{|\underline{T}|} \int_{\X}\prod_{j: T_j = h} \kappa(x_j ; \theta) G_0(\dd \theta)\right],
\end{equation}
where the expectation is taken with respect to the indicators $\underline{T} = (T_1, \ldots, T_k)$ with distribution
\begin{equation}\label{eq:dist_T}
\prob(\underline{T} = \underline{t}) \propto  \prod_{h=1}^{|\underline{t}|} \int_{\R_+} e^{-\gamma} \gamma^{n_h} \rho(\dd \gamma) .    
\end{equation}
\end{theorem}
The expression of the Janossy density in Theorem \ref{thm:janossy_sncp} naturally enables maximum likelihood estimation for $\Phi \sim \textsc{sncp}(\kappa, \nu)$, obtained by selecting the parameters that maximize $j_k$ in \eqref{eq:janossy_fact}--\eqref{eq:dist_T}.
In particular, the structure revealed by \Cref{thm:janossy_sncp} suggests the potential use of expectation-maximization algorithms for parameter estimation, thereby providing a viable alternative to existing inferential methods \citep{MoWaBook03}. This line of research is currently under active investigation, and it will be the object of future research studies.

\section{Discussion}

The implications of Theorem \ref{teo:main} and its generalizations extend well beyond the two concrete applications presented in this work. As discussed in Section~\ref{sec:summary_stat}, our main result enables the evaluation of summary statistics for both stationary and non-stationary point processes, which is useful for both exploratory analysis and model checking via, e.g., envelope tests  \citep{Chiu2013}.

We also envision applications in the field of Bayesian nonparametrics, where superpositions of point processes are used to define prior distributions when data are divided into groups \citep{Nip14, griffin2013comparing}: Theorem \ref{teo:main} can be leveraged for posterior analysis and numerical computations in such models.
Finally, Theorem~\ref{thm:red_palm_sncp} plays a key role in deriving posterior representations in settings where the \textsc{sncp} serves as a Bayesian nonparametric prior. Relevant examples include Bayesian mixture modeling \citep{Beraha2025JRSSB} and extended feature allocation models, namely an extension of Bayesian clustering models that allows individuals to share multiple features \citep{beraha2025extended}.

\section*{Acknowledgments}
{
MB gratefully acknowledges support from the Italian Ministry of Education, University and Research (MUR), “Dipartimenti di Eccellenza” grant 2023-2027.
FC gratefully acknowledges support of MUR - Prin 2022 - Grant no. 2022CLTYP4, funded by the European Union – Next Generation EU. LG gratefully acknowledges support from the Office of Naval Research (N00014-24-1-2626-P00002).}

\section*{Appendix}
\section*{Proof of Theorem \ref{teo:main}}

Let $\Phi = \Phi_1 + \Phi_2$, with $\Phi_1$ independent of $\Phi_2$. By \citep[Proposition 3.2.1]{BaBlaKa}, the Palm distribution of a point process $\Phi$ is uniquely characterized by the following relation: for all measurable $f, g: \X \rightarrow \R_+$ such that $g$ is $M_{\Phi}$-integrable
    \begin{equation}\label{eq:der_thm1}
        \partder{t} \mathcal{L}_{\Phi}(f + t g)|_{t = 0} =  - \E\left[ \Phi(g) e^{-\Phi(f)} \right] =  - \int_{\X} g(x) \mathcal{L}_{\Phi_x}(f) M_{\Phi}(\dd x),
    \end{equation}
    where $\mathcal{L}_{\Phi}(f)$ denotes the Laplace functional of $\Phi$ evaluated at $f$, i.e., $\mathcal{L}_{\Phi}(f) = \E\left[e^{- \int_{\X} f(x) \Phi(\dd x)}\right]$, and $\Phi(f) := \int_{\X} f(x) \Phi(\dd x)$.
    Therefore, we have
    \begin{align*}
        \E\left[ \Phi(g) e^{-\Phi(f)} \right] &= \E\left[ \int_{\X} g(x) \e^{-\int_{\X} f(y) (\Phi_1 + \Phi_2)(\dd y)} \Phi_1(\dd x) \right] \\
        & \qquad + \E\left[ \int_{\X} g(x) \e^{-\int_{\X} f(y) (\Phi_1 + \Phi_2)(\dd y)} \Phi_2(\dd x) \right] \\
        & = \E\left[\Phi_1(g) \e^{-\Phi_1(f)}\right]\E[\e^{-\Phi_2(f)}] +\E\left[\Phi_2(g) \e^{-\Phi_2(f)}\right]\E[\e^{-\Phi_1(f)}]
    \end{align*}
    where the last equality follows from the independence of $\Phi_1$ and $\Phi_2$.
    Then, applying again \citep[Proposition 3.2.1]{BaBlaKa}, we obtain
    \begin{equation*}
       \E\left[ \Phi(g) e^{-\Phi(f)} \right] = 
        \left[\int_{\X} g(x) \mathcal{L}_{\Phi_{1x}}(f) M_{\Phi_1}(\dd x)\right] \mathcal{L}_{\Phi_2}(f) + \left[\int_{\X} g(x) \mathcal{L}_{\Phi_{2x}}(f) M_{\Phi_2}(\dd x)\right] \mathcal{L}_{\Phi_1}(f).
    \end{equation*}
Since $M_\Phi(\dd x) = M_{\Phi_1}(\dd x ) + M_{\Phi_2}(\dd x)$, and $M_{\Phi_i}$ is absolutely contintuous with respect to $M_\Phi$, we write
 \begin{equation*}
       \E\left[ \Phi(g) e^{-\Phi(f)} \right] = 
        \int_{\X} g(x) \left\{ \mathcal{L}_{\Phi_2}(f) \mathcal{L}_{\Phi_{1x}}(f) \frac{\dd M_{\Phi_1}}{\dd M_{\Phi}}(x) + \mathcal{L}_{\Phi_1}(f) \mathcal{L}_{\Phi_{2x}}(f) \frac{\dd M_{\Phi_2}}{\dd M_{\Phi}}(x)\right\} M_\Phi(\dd x).
    \end{equation*}  
    Therefore, from \eqref{eq:der_thm1}, it holds that
    \[
    \mathcal{L}_{\Phi_x}(f) = \mathcal{L}_{\Phi_2}(f) \mathcal{L}_{\Phi_{1x}}(f) \frac{\dd M_{\Phi_1}}{\dd M_{\Phi}}(x) + \mathcal{L}_{\Phi_1}(f) \mathcal{L}_{\Phi_{2x}}(f) \frac{\dd M_{\Phi_2}}{\dd M_{\Phi}}(x),
    \]
    which is an alternative writing for the statement of the theorem.

\bibliography{references}

\clearpage

\setcounter{equation}{0}
\renewcommand\theequation{S\arabic{equation}}
\renewcommand\theHequation{S\arabic{equation}}

\setcounter{theorem}{0}
\renewcommand\thetheorem{S\arabic{theorem}}
\renewcommand\theHtheorem{S\arabic{theorem}}

\setcounter{corollary}{0}
\renewcommand\thecorollary{S\arabic{corollary}}
\renewcommand\theHcorollary{S\arabic{corollary}}

\setcounter{proposition}{0}
\renewcommand\theproposition{S\arabic{proposition}}
\renewcommand\theHproposition{S\arabic{proposition}}

\setcounter{lemma}{0}
\renewcommand\thelemma{S\arabic{lemma}}
\renewcommand\theHlemma{S\arabic{lemma}}

\setcounter{figure}{0}
\renewcommand\thefigure{S\arabic{figure}}
\renewcommand\theHfigure{S\arabic{figure}}

\setcounter{table}{0}
\renewcommand\thetable{S\arabic{table}}
\renewcommand\theHtable{S\arabic{table}}

\setcounter{section}{0}
\renewcommand{\thesection}{S\arabic{section}}
\renewcommand{\theHsection}{S\arabic{section}}

\begin{center}
   \LARGE Supplementary Material for:\\
    ``Palm distributions of superposed point processes for statistical inference''
\end{center}

This Supplementary Material contains mathematical proofs and additional methodological details supporting the accompanying paper.

\section{Mathematical background on point processes} \label{app:point_process_definition}

Here, we provide the formal definition of a point process. A rigorous treatment can be found in \cite{DaVeJo2} or \cite{BaBlaKa}. See also \cite{Kallenberg84} for some further intuition. Following the notation in \cite{BaBlaKa}, let $\X$ be a Polish space equipped with the corresponding Borel $\sigma$-algebra $\Xcr$.  
Let $\M_{\X}$ be the set of counting measures on $(\X, \Xcr)$, that is  $\nu \in \M_{\X}$ if (i) $\nu$ is a locally finite measure, i.e., $\nu(B) < \infty$ for all relative compact sets $B \in \Xcr$, and (ii) $\nu(B) \in \{0,1,\ldots\}$ for all relatively compact sets $B \in \Xcr$. Let $\mathcal M_{\X}$ be the smallest $\sigma$-algebra which makes the mappings $\nu \mapsto \nu(B)$ measurable, for any $B \in \Xcr$. By definition, a point process $\Phi$ on the space $\X$ is a measurable map from an underlying probability space $(\Omega, \Acr, \prob)$ taking values in $(\M_{\X}, \mathcal M_{\X})$. Its probability distribution is given by $\plaw_\Phi = \prob \circ \Phi^{-1}$. Notably, any point process $\Phi$ can be represented as $\Phi = \sum_{j\geq 1} \delta_{X_j}$, where $(X_j)_{j\geq 1}$ is a sequence of random variables taking values in $\X$, and  $\delta_{x}$ denotes the Dirac delta mass at $x$.
To describe the probabilty distribution of $\Phi$, the most essential summary is its mean measure $M_{\Phi}$, which is defined as $M_{\Phi}(B) = \E[\Phi(B)]$ for any $B \in \Xcr$. In general, the $k$-th order factorial moment measure $M_{\Phi}^{(k)}$ of $\Phi$ is the mean measure of the $k$-th factorial power of $\Phi$, i.e., of the point process $\Phi^{(k)}$  defined as:
\[
    \Phi^{(k)} := \sum_{(j_1, \ldots, j_k)}^{\not =}  \delta_{(X_{j_1}, \ldots, X_{j_k})},
\]
where the symbol $\not =$ means that the sum is over all pairwise distinct indexes. Informally, $M_{\Phi}^{(k)}(\dd x_1\,\ldots\, \dd x_k)$ corresponds to the probability that $\Phi$ has atoms in infinitesimal neighborhoods of $x_j$, $j=1, \ldots, k$. See \cite{BaBlaKa} for a detailed account.

To define the \emph{Palm distributions}, we introduce the Campbell measure of $\Phi$ by $\mathscr C_{\Phi}(B \times L) = \E[\Phi_B \indicator(\Phi \in L)]$, for  $B \in \Xcr$ and  $L \in \calM_{\X}$.
Under the assumption that $M_\Phi$ is $\sigma$-finite, then $\mathscr C_{\Phi}$ admits a disintegration
\[
   \mathscr C_{\Phi}(B \times L) = \int_B \plaw_{\Phi}^x(L) M_\Phi(\dd x),
\]
where $\{\plaw_{\Phi}^x\}_{x \in \X}$ is the (a.s.) unique disintegration probability kernel of $\mathscr C_\Phi$ with respect to $M_\Phi$, and it is referred to as the Palm kernel of $\Phi$.
For fixed $x \in \X$, $\plaw_{\Phi}^x$ is a probability distribution over $(\M_{\X}, \calM_{\X})$, named the Palm distribution of $\Phi$ at $x$, and thus it can be identified with the law of a point process $\Phi_x \sim \plaw_{\Phi}^x$, which is consequently called a \emph{Palm version} of $\Phi$ at $x$. Since $\Pp(\Phi_x(\{x\}) \geq  1) = 1$, i.e., $x$ is a trivial atom of $\Phi_x$, one can define the point process $\Phi^!_{x}:  = \Phi_{x} - \delta_{x}$, which is called 
the \textit{reduced Palm version of $\Phi$ at $x \in \X$}, whose associated reduced Palm kernel is indicated by $\plaw_{\Phi}^{!x}$.

In a similar fashion, under the assumption that the $k$-th factorial moment measure $M_\Phi^{(k)}$ is $\sigma$-finite, it is possible to construct the family of  $k$-th order Palm distributions $\{\plaw_{\Phi}^{\underline{x}}\}_{\underline{x} \in \X^k}$, and the generic probability measure $\plaw_{\Phi}^{\underline{x}}$ can be interpreted as the distribution of $\Phi$ given that $\underline{x} = (x_1, \ldots, x_k)$ are atoms of $\Phi$. 
Again, by removing the trivial atoms $(x_1, \ldots , x_k)$, we obtain the reduced Palm distributions $\plaw_{\Phi}^{!\underline{x}}$, namely the probability law of
\[
\Phi^!_{\underline{x}} : = \Phi_{\underline{x}} -\sum_{j =1}^k \delta_{x_j}.
\]

\section{General results on Palm distributions of point processes}

In the present section we prove two general results concerning point process theory that will be useful in the sequel but that could be also of independent interest.
First, we characterize the reduced Palm distribution of a mixture of point processes, which is a mixture of the reduced Palm distributions of the individual components. 
\begin{proposition}[Palm distributions of mixtures of point processes]\label{prp:palm_mixture}
    Let $\plaw_{\Phi_i}$, $i=1, \ldots, m$, be probability distributions over $(\mathbb M_{\X}, \mathcal{M}_{\X})$ and $p_1, \ldots, p_m$ such that $p_i \geq 0$, $\sum_{i=1}^m p_i = 1$. Define $\Phi$ such that $\plaw_{\Phi} = \sum_{i=1}^m p_i \plaw_{\Phi_i}$ and $\Phi_i$ with law $\plaw_{\Phi_i}$. 
    Then
    \[
        \plaw^{!x}_{\Phi} = \sum_{i=1}^m w_i(x) \plaw^{!x}_{\Phi_i}, \qquad w_i(x) = p_i \frac{\dd M_{\Phi_i}}{\dd M_{\Phi}}(x) .
    \]
\end{proposition}
\begin{proof}
    By the Campbell-Little-Mecke (\textsc{clm}) formula
    \[
         \E \int f(x, \Phi - \delta_{x}) \Phi(\dd x)
=\sum_{i=1}^m p_i \int f(x,\varphi) \plaw^{!x}_{\Phi_i}(\dd\varphi) M_{\Phi_i}(\dd x).
    \]
    Since $M_{\Phi}$ clearly dominates all the $M_{\Phi_i}$'s we have $M_{\Phi_i}(\dd x) = \frac{\dd M_{\Phi_i}}{\dd M_{\Phi}}(x) M_{\Phi}(\dd x)$ and the proof follows.
\end{proof}

In the second result, we provide a general description of the Janossy measures of finite point processes in terms of their Palm distributions and their factorial moment measures. Remind that, for any finite point process $\Phi$, the family of Janossy measures $J_k(\cdot)$, $k\geq 0$, characterizes its probability distribution \citep[Proposition 5.3.II,][]{DaVeJo1}. This result can also be found in \cite{coeurjolly2017tutorial}.

\begin{theorem}[Janossy measures of finite point processes]\label{thm:janossy_general}
Let $\Phi$ be a finite point process on $\X$. Then, its Janossy measures satisfy
\[
J_k(\dd x_1\,\ldots\, \dd x_k) = \prob(\Phi^!_{(x_1,\ldots,x_k)}(\X\setminus \{x_1,\ldots,x_k\} ) = 0)  M^{(k)}_{\Phi}(\dd x_1\,\ldots\, \dd x_k).
\]
Moreover, if $\Phi$ is a simple point process, it simplifies as
\[
J_k(\dd x_1\,\ldots\, \dd x_k) = \prob(\Phi^!_{(x_1,\ldots,x_k)}(\X) = 0)  M^{(k)}_{\Phi}(\dd x_1\,\ldots\, \dd x_k).
\]
\end{theorem}
\begin{proof}
Let $x_1,\ldots,x_k$ be a generic set of points and denote with $\dd x_j$ the ball of radius $\epsilon$ and center $x_j$, as $j = 1,\ldots,k$. Take the radius $\epsilon$ small enough so that the balls are disjoint. Denote with $\tilde{\X} = \X \setminus \cup_{j=1}^k \dd x_j$. In this case, the Janossy measure correspond to
\[
\begin{split}
J_k(\dd x_1\,\ldots\,\dd x_k) & = \prob\left( \Phi(\dd x_1) = 1,\ldots, \Phi(\dd x_k ) = 1, \Phi(\tilde{\X}) = 0 \right) \\
&= \E\left[ \prod_{j=1}^k \indicator_{\{1\}}(\Phi(\dd x_j)) \times \indicator_{\{0\}} (\Phi(\tilde{\X})) \right].
\end{split}
\]
Proceed with the following computations:
\begin{equation}\label{eq:janossy_start}
\begin{aligned}
    J_k(\dd x_1\,\ldots\,\dd x_k) &= \E\left[ \prod_{j=1}^k \indicator_{\{1\}}(\Phi(\dd x_j)) \times \indicator_{\{0\}} (\Phi(\tilde{\X})) \right]\\
    &= \E\left[ \int_{\X^k} \indicator_{\{0\}} (\Phi(\tilde{\X})) \prod_{j=1}^k \delta_{y_j} (\dd x_j) \Phi^k(\dd y_1\,\ldots\, \dd y_k)  \right],
\end{aligned}
\end{equation}
where $\Phi^k$ is the $k$-th power of $\Phi$. Note that the term $\prod_{j=1}^k \delta_{y_j}(\dd x_j)$ entails that the integrand is zero on sets of the type
\[
    \{\underline{y} \in \X^k : \; y_j = y_\ell \text{ for } j \neq \ell \}.
\]
Therefore we can replace $\Phi^k$ with the $k$-th factorial power $\Phi^{(k)}$. Moreover, observe that for all $\underline{y}$ such that the integrand is non zero, $\Phi(\tilde{\X}) = (\Phi - \sum_{j=1}^k \delta_{y_j}) (\tilde{\X})$. Continuing from \eqref{eq:janossy_start}, we get
\begin{align}
        J_k(\dd x_1\,\ldots\,\dd x_k) &= \E\left[ \int_{\X^k} \indicator_{\{0\}} ((\Phi- \sum_{j=1}^k \delta_{y_j})(\tilde{\X})) \prod_{j=1}^k \delta_{y_j} (\dd x_j) \Phi^{(k)}(\dd y_1\,\ldots\, \dd y_k)  \right] \notag\\
        &= \int_{\X^k} \E\left[ \indicator_{\{0\}} ( \Phi^!_{\underline{y}}(\tilde{\X}))\right] \prod_{j=1}^k \delta_{y_j} (\dd x_j) M^{(k)}_{\Phi}(\dd y_1\,\ldots\, \dd y_k) \notag\\
        &= \E\left[ \indicator_{\{0\}} ( \Phi^!_{\underline{x}}(\tilde{\X}))\right] M^{(k)}_{\Phi}(\dd x_1\,\ldots\, \dd x_k), 
\end{align}
where the second equality follows from the \textsc{clm} formula.     
\end{proof}

\section{Proofs of the main results and extensions}

\subsection{Extensions of \Cref{teo:main}}
\label{app:extensions}

In this section, we present and discuss two natural extensions to the main theoretical result of the paper in \Cref{teo:main}. Firstly, we provide the characterization of the Palm distributions of the superposition of multiple independent point
processes. This is shown in the next corollary.
\begin{corollary}\label{cor:palm_superposition_multiple_processes}
     Let $\Phi_1, \ldots, \Phi_m$ be independent point processes on $\X$, then
    \[
        \left(\sum_{h=1}^m \Phi_h\right)_x \deq \Phi_{i x} + \sum_{q \neq i} \Phi_q \quad \text{with probability proportional to } M_{\Phi_i}(\dd x), \quad \text{for } i=1,\ldots,m.
    \]
    Moreover, the same distributional equivalence holds true for the corresponding reduced Palm versions.
\end{corollary}
\begin{proof}
    The proof is analogous to the one of \Cref{teo:main}, reported in the main paper.
\end{proof}

Secondly, the Palm distributions of the superposition of two independent processes under multiple conditioning points can be formally addressed by applying the Palm algebra. Under some technical conditions on the factorial moment measures of $\Phi$ (see Section 3.3.2 in \cite{BaBlaKa}), Palm algebra allows to write
\begin{equation}\label{eq:palm_algebra}
    \left(\Phi^!_{\underline{x}}\right)^!_{\underline{y}} \deq \Phi^!_{(\underline{x}, \underline{y})}, \quad (\underline{x}, 
    \underline{y}) \in \X^k \times \X^\ell.
\end{equation}
As a warm up to the proof of Theorem \ref{teo:multiple}, we investigate the special case of the Palm distributions of the superposition of two independent processes under two conditioning points. The following corollary holds true.

\begin{corollary}
Let $\Phi_1$ and $\Phi_2$ be two independent point processes, and $\Phi= \Phi_1+\Phi_2$ the corresponding superposition. For any $(x, y) \in \X^2$, we have
    \[
    (\Phi_1 + \Phi_2)^!_{(x, y)} \deq \begin{cases}
        (\Phi_1)^!_{(x, y)} + \Phi_2 &\text{with probability equal to } \displaystyle{\frac{ \dd M_{\Phi_1}^{(2)}}{\dd M_{\Phi}^{(2)} } (x,y)}\\
        (\Phi_1)^!_{x} +  (\Phi_2)^!_{y} &\text{with probability equal to }  \displaystyle{\frac{\dd M_{\Phi_{1}}  \times M_{\Phi_{2}}}{\dd M_{\Phi}^{(2)} } (x,y)}\\
       (\Phi_1)^!_{y} +  (\Phi_2)^!_{x} &\text{with probability equal to }   \displaystyle{\frac{\dd M_{\Phi_{1}} \times M_{\Phi_{2}}}{\dd M_{\Phi}^{(2)} } (y,x)}\\
        \Phi_1+ (\Phi_2)^!_{(x, y)} &\text{with probability equal to }  \displaystyle{\frac{\dd M_{\Phi_2}^{(2)}}{\dd M_{\Phi}^{(2)} } (x,y) }\\
    \end{cases}
    \]
    where $M_{\Phi_1} \times M_{\Phi_2}$ denotes the product measure.
\end{corollary}
\begin{proof}
By Palm algebra in \eqref{eq:palm_algebra}, we have that 
\[(\Phi_1 + \Phi_2)^!_{(x, y)}=((\Phi_1 + \Phi_2)^!_x )_y^{!}.
\]
From \Cref{teo:main}, $(\Phi_1 + \Phi_2)^!_x $ is characterized by the following mixture of point processes,
\begin{equation}\label{eq:psi_mixture}
     (\Phi_1 + \Phi_2)^!_{x} \deq \begin{cases}
            (\Phi_1)^!_{ x} + \Phi_2 & \text{with probability proportional to }  M_{\Phi_1}(\dd x) \\
            \Phi_{1} + (\Phi_2)^!_{x} & \text{with probability proportional to } M_{\Phi_2}(\dd x).
        \end{cases}
\end{equation}
For notational clarity in the following of the proof, let $\Psi$ denote $(\Phi_1 + \Phi_2)^!_{x}$, let $\Psi_1$ denote $(\Phi_1)^!_{ x} + \Phi_2$ and let $\Psi_2$ denote $\Phi_{1} + (\Phi_2)^!_{x}$. The goal is then to characterize the distribution of $\Psi^!_y$. 

First, observe that the mixture representation in \eqref{eq:psi_mixture} can be written in terms of the probability laws  of the processes as
\[
\plaw_{\Psi} = \frac{\dd M_{\Phi_1}}{\dd M_{\Phi}} (x) \plaw_{\Psi_1} + \frac{\dd M_{\Phi_2}}{\dd M_{\Phi}} (x) \plaw_{\Psi_2}.
\]
Second, the law of $\Psi^!_y$ is characterized by applying \Cref{prp:palm_mixture}. Specifically, it holds that
\begin{equation}\label{eq:law_psi_y}
    \plaw^{!y}_{\Psi} = \frac{\dd M_{\Phi_1}}{\dd M_{\Phi}} (x) \frac{\dd M_{\Psi_1}}{\dd M_{\Psi}} (y) \plaw^{!y}_{\Psi_1} + \frac{\dd M_{\Phi_2}}{\dd M_{\Phi}} (x) \frac{\dd M_{\Psi_2}}{\dd M_{\Psi}} (y) \plaw^{!y}_{\Psi_2}.
\end{equation}
Now, we need to determine the laws of $(\Psi_1)^!_y$ and $(\Psi_2)^!_y$. Still applying \Cref{teo:main}, we have
\[
\plaw^{!y}_{\Psi_1} = \frac{\dd M_{(\Phi_1)^!_x}}{\dd M_{\Psi_1}} (y) \plaw_{(\Phi_1)^!_{(x,y)} + \Phi_2} + \frac{\dd M_{\Phi_2}}{\dd M_{\Psi_1}} (y) \plaw_{(\Phi_1)^!_{x} + (\Phi_2)^!_{y}},
\]
and 
\[
\plaw^{!y}_{\Psi_2} = \frac{\dd M_{\Phi_1}}{\dd M_{\Psi_2}} (y) \plaw_{(\Phi_1)^!_{y} + (\Phi_2)^!_x} + \frac{\dd M_{(\Phi_2)^!_x}}{\dd M_{\Psi_2}} (y) \plaw_{\Phi_1 + (\Phi_2)^!_{(x,y)}}.
\]
Finally, plugging in the last two expressions in \eqref{eq:law_psi_y}, we obtain
\begin{equation*}
    \begin{aligned}
        \plaw^{!y}_{\Psi} &= \frac{\dd M_{\Phi_1}}{\dd M_{\Phi}} (x) \left[ \frac{\dd M_{(\Phi_1)^!_x}}{\dd M_{\Psi}} (y) \plaw_{(\Phi_1)^!_{(x,y)} + \Phi_2} + \frac{\dd M_{\Phi_2}}{\dd M_{\Psi}} (y) \plaw_{(\Phi_1)^!_{x} + (\Phi_2)^!_{y}} \right]\\
        &\quad + \frac{\dd M_{\Phi_2}}{\dd M_{\Phi}} (x) \left[  \frac{\dd M_{\Phi_1}}{\dd M_{\Psi}} (y) \plaw_{(\Phi_1)^!_{y} + (\Phi_2)^!_x} + \frac{\dd M_{(\Phi_2)^!_x}}{\dd M_{\Psi}} (y) \plaw_{\Phi_1 + (\Phi_2)^!_{(x,y)}}\right],
    \end{aligned}
\end{equation*}
which is equivalent to write
\[
    \Psi^!_y \deq \begin{cases}
        (\Phi_1)^!_{(x, y)} + \Phi_2 &\text{with probability }\displaystyle{\frac{\dd M_{\Phi_1}}{\dd M_{\Phi}} (x) \frac{\dd M_{(\Phi_1)^!_x}}{\dd M_{\Psi}} (y)} \\
        (\Phi_1)^!_{x} +  (\Phi_2)^!_{y} &\text{with probability }  \displaystyle{\frac{\dd M_{\Phi_1}}{\dd M_{\Phi}} (x)  \frac{\dd M_{\Phi_2}}{\dd M_{\Psi}} (y)}\\
       (\Phi_1)^!_{y} +  (\Phi_2)^!_{x} &\text{with probability }   \displaystyle{\frac{\dd M_{\Phi_2}}{\dd M_{\Phi}} (x)  \frac{\dd M_{\Phi_1}}{\dd M_{\Psi}} (y)}\\
        \Phi_1+ (\Phi_2)^!_{(x, y)} &\text{with probability }  \displaystyle{\frac{\dd M_{\Phi_2}}{\dd M_{\Phi}} (x) \frac{\dd M_{(\Phi_2)^!_x}}{\dd M_{\Psi}} (y)} \\
    \end{cases}
\]
The thesis follows from noticing that $M_{\Phi_1}^{(2)}(\dd x \, \dd y) = M_{(\Phi_1)^!_{x}}(\dd y) M_{\Phi_1}(\dd x)$ and that $M_{\Phi} (\dd x)M_{\Psi}(\dd y) = M_{\Phi} (\dd x)M_{\Phi_x^{!}}(\dd y) = M_{\Phi}^{(2)} (\dd x \, \dd y)$, thanks to \citep[Proposition 3.3.9]{BaBlaKa}. 
\end{proof} 

We finally  prove Theorem \ref{teo:multiple}, presented in the paper.
\begin{proof}[of Theorem \ref{teo:multiple}]
We proceed by induction on the number of points $k$. Observe that for $k=1$, the statement is a simple rewriting of Corollary \ref{cor:palm_superposition_multiple_processes}.
Assume now that the statement holds for $k-1\geq 1$ points, and we show it consequently holds for $k$ points. Fix $\underline{x}=(x_1,\ldots,x_k)$ with pairwise distinct coordinates, and write
\[
\underline{x}^-=(x_1,\ldots,x_{k-1}),\qquad \underline{T}^-=(T_1,\ldots,T_{k-1}).
\]
Moreover, we write $\underline{x}_i(\underline{t}^-) = (x_j\colon j=1, \ldots, k-1, \ t_j = i)$ to stress the dependence on $\underline{t}^-$.
By Palm algebra, for $M_\Phi^{(k)}$-almost all $\underline{x}$, we can write
\begin{equation}
\label{eq:palm-algebra-n}
\Phi^!_{\underline{x}}\ \deq\ \left(\Phi^!_{\underline{x}^-}\right)^!_{x_k}.
\end{equation}
By the induction hypothesis, for $M_\Phi^{(k-1)}$-almost all $\underline{x}^-$, the following mixture representation holds true
\[
\plaw_{\Phi^{!}_{{\underline{x}^-}}}
=\sum_{\underline{t}^-\in\{1,\ldots,m\}^{k-1}} \pi_{\underline{t}^-}(\underline{x}^-)\,\plaw_{\Psi_{\underline{t}^-}},
\quad \text{where} \quad
\Psi_{\underline{t}^-}:=\sum_{i=1}^m (\Phi_i)^!_{\underline{x}_i(\underline{t}^-)},
\]
moreover, conditionally on $\underline{T}^-=\underline{t}^-$, the processes $(\Phi_i)^!_{\underline{x}_i(\underline{t}^-)}$ are mutually independent (since the $\Phi_i$'s are independent), and
\begin{equation}
\label{eq:ind-hyp-weights}
\pi_{\underline{t}^-}(\underline{x}^-)=P(\underline{T}^-=\underline{t}^-)\propto \prod_{i=1}^m  \rho_{\Phi_i}^{(n_i(\underline{t}^-))}\left(\underline{x}_i(\underline{t}^-)\right)
\end{equation}
where we have made explicit the dependence on $\underline{x}^-$. Observe that the $\pi_{\underline{t}^-}(\underline{x}^-)$'s are normalized and sum to $1$ as $\underline{t}^-$ varies in $\{1, \ldots, m\}^{k-1}$.

We now take the reduced Palm distribution at the additional point $x_k$. Applying Proposition \ref{prp:palm_mixture} to the above mixture yields
\begin{equation}
\label{eq:mix-after-propS1}
\mathsf{P}_{\left(\Phi^!_{\underline{x}^-}\right)^!_{x_k}}
=\sum_{\underline{t}^-\in\{1,\ldots,m\}^{k-1}} \widetilde{\pi}_{\underline{t}^-}(\underline{x})\,\plaw^{!x_k}_{\Psi_{\underline{t}^-}},
\qquad
\widetilde{\pi}_{\underline{t}^-}(\underline{x})=\pi_{\underline{t}^-}(\underline{x}^-)\,\frac{\dd M_{\Psi_{\underline{t}^-}}}{\dd M_{\Phi^!_{\underline{x}^-}}}(x_k).
\end{equation}
Observe that also the $\widetilde{\pi}_{\underline{t}^-}(\underline{x})$'s sum to 1 over $\underline{t}^-$.

Next, fix $\underline{t}^-$ and analyze $\plaw^{!x_k}_{\Psi_{\underline{t}^-}}$.
Since $\Psi_{\underline{t}^-}$ is the superposition of the independent processes $(\Phi_i)^!_{\underline{x}_i(\underline{t}^-)}$, as $i=1, \ldots, m$, Corollary \ref{cor:palm_superposition_multiple_processes} applied to
$\sum_{i=1}^m (\Phi_i)^!_{\underline{x}_i(\underline{t}^-)}$ implies that $\plaw^{!x_k}_{\Psi_{\underline{t}^-}}$ is a mixture over $i\in\{1,\ldots,m\}$ with
mixing weights equal to
\[
\frac{\dd M_{(\Phi_i)^!_{\underline{x}_i(\underline{t}^-)}}}{\dd M_{\Psi_{\underline{t}^-}}}(x_k), 
\]
and mixture component equal (by Palm algebra within
$\Phi_j$) to
\[
\left((\Phi_i)^!_{\underline{x}_i(\underline{t}^-)}\right)^!_{x_k}\ +\ \sum_{q\neq i} (\Phi_q)^!_{\underline{x}_q(\underline{t}^-)}
\ \deq\
(\Phi_i)^!_{(\underline{x}_i(\underline{t}^-),x_k)}\ +\ \sum_{q\neq i} (\Phi_q)^!_{\underline{x}_q(\underline{t}^-)}.
\]
Equivalently, if we define the extended allocation $\underline{t}\in\{1,\ldots,m\}^k$ by $t_j=t^-_j$ for $j\leq k-1$ and $t_k=i$, then the
corresponding component can be written compactly as 
\[
\sum_{i=1}^m (\Phi_i)^!_{\underline{x}_i(\underline{t})}.
\]
Moreover, conditional independence of the summands holds because each summand is a measurable function of $\Phi_i$ alone and the
$\Phi_i$'s are independent.

Combining \eqref{eq:palm-algebra-n}--\eqref{eq:mix-after-propS1} with the above mixture for each $\plaw^{!x_k}_{\Psi_{\underline{t}^-}}$, we obtain a
mixture representation for $\Phi^!_{\underline{x}}$ indexed by full allocations $\underline{t}\in\{1,\ldots,m\}^k$, with mixture component
$\sum_{i=1}^m (\Phi_i)^!_{\underline{x}_i(\underline{t})}$.

It remains to compute the resulting weights and show they match the product of factorial-moment densities.
Let $\underline{t}\in\{1,\ldots,m\}^k$ and write $\underline{t}^-=(t_1,\ldots,t_{k-1})$ and $t_k=i$.
Up to a normalizing constant (depending on $\underline{x}$ only), the weight assigned to $\underline{t}$ is proportional to 
\begin{equation}
\label{eq:weight-decomp}
\pi_{\underline{t}^-}(\underline{x}^-)\,M_{(\Phi_i)^!_{\underline{x}_i(\underline{t}^-)}}(\dd x_k).
\end{equation}

Under the assumption that $M_{\Phi_i}^{(r)}$ admits a density $\rho_{\Phi_i}^{(r)}$ with respect to $\mu^{\otimes r}$, Proposition 3.3.9 in \cite{BaBlaKa} gives, for $\mu$-almost all $y$,
\[
M_{(\Phi_i)^!_{\underline{x}_i(\underline{t}^-)}}(\dd y)
=\frac{\rho_{\Phi_i}^{(n_i(\underline{t}^-)+1)}\left(\underline{x}_i(\underline{t}^-),y\right)}{\rho_{\Phi_i}^{(n_i(\underline{t}^-))}\left(\underline{x}_i(\underline{t}^-)\right)}\mu(\dd y),
\]
with the convention $\rho_{\Phi_i}^{(0)}(\emptyset)=1$ when $n_i(\underline{t}^-)=0$.
Plugging this into \eqref{eq:weight-decomp} and using \eqref{eq:ind-hyp-weights}, we obtain that the weight assigned to $\underline{t}$ si proportional to
\[
\left(\prod_{h=1}^m \rho_{\Phi_h}^{(n_h(\underline{t}^-))}\left(\underline{x}_h(\underline{t}^-)\right)\right)
\frac{\rho_{\Phi_i}^{(n_i(\underline{t}^-)+1)}\left(\underline{x}_i(\underline{t}^-),x_k\right)}{\rho_{\Phi_i}^{(n_i(\underline{t}^-))}\left(\underline{x}_i(\underline{t}^-)\right)}.
\]
Simple algebra allows us to rewrite the expression above as
\[
\prod_{h\neq i} \rho_{\Phi_h}^{(n_h(\underline{t}))}\left(\underline{x}_h(\underline{t})\right)\,
\rho_{\Phi_i}^{(n_i(\underline{t}))}\left(\underline{x}_i(\underline{t})\right)
=
\prod_{i=1}^m \rho_{\Phi_i}^{(n_i(\underline{t}))}\left(\underline{x}_i(\underline{t})\right),
\]
which is exactly the claimed expression.
\end{proof}

\section{Summary statistics for corrupted point processes}

In the present section we provide other summary statistics for corrupted point processes discussed in Section \ref{sec:corruzione}. For the purpose of illustration we focus on $\X = \R^2$, and we assume to observe the corrupted point process $\Phi = \Phi_1+ \Phi_2$ , where $\Phi_1$ is the point process of interest, while $\Phi_2$ represents a random background noise, i.e., a homogeneous Poisson point process with intensity $\rho_2$. We further assume that the two point processes are independent. We also focus on different choices of $\Phi_1$, in particular we focus on  the determinantal point process \citep{Lavancier2015} and the Mat\'ern cluster process \citep{BaBlaKa}. \\

\subsection{Inhomogeneous summary statistics} \label{app:nonstat_summary}

In the sequel we assume that $\Phi_1$ is a general point process with intensity measure $M_{\Phi_1} (\dd x) = m_{1} (x) \dd x $ on $\R^2$. Thus, we also denote by $m(x)$ the intensity function of the superposed point process $\Phi$, which equals $m_{1}(x) + \rho_2$.

We define the inhomogeneous $K$-function and $G$-function. According to \cite{BaddeleyRubakTurner2015}, 
inhomogeneous $K$-function is defined as 
\begin{equation}
    \label{eq:ino_K_def}
    K_{\Phi}^{in} (r, x) = \E \left[ \int_{B (x, r) } \frac{1}{m (u)} \Phi^{!}_x (\dd u)\right]
\end{equation}
See \cite[pg. 243]{BaddeleyRubakTurner2015} for the assumptions that are typically required on $\Phi$ to define the inhomogeneous $K$-function $K^{in}_\Phi $. 
Besides, the inhomogeneous $G$-function has been introduced by \cite{Lieshout11}:
\begin{equation}
    \label{eq:ino_G_def}
    G^{in}_{\Phi}(r,x) = 1- \E \left[  \prod_{X \in \Phi^{!}_x} \left(1- \frac{\bar m}{m (X)} \indicator_{B (x, r)}(X)  \right)\right]
 \end{equation}
 where we have defined $\bar m := \inf_{x \in \R^2}  m(x) $ and we suppose $\bar m >0$.
See also \cite[pg. 277]{BaddeleyRubakTurner2015} for additional details and typical assumptions on the point process $\Phi$. 

\begin{proposition}
    Let $\Phi= \Phi_1+ \Phi_2$ be a corrupted point process on $\R^2$, where $\Phi_1$ is a point process with intensity function $m_1$ and $\Phi_2$ is a homogeneous Poisson point process with intensity $\rho_2$. Suppose that $\Phi_1$ and $\Phi_2$ are independent point processes. Thus, the inhomogeneous $K$-function equals
    \begin{equation}
        K_{\Phi}^{in} (r,x)  =\frac{m_{1} (x) }{m_{1} (x)+ \rho_2} \int_{B(x, r)} \frac{m_{1x}^! (u) - m_1 (u) }{m_{1}(u)+ \rho_2} \dd u  + \pi r^2\label{eq:K_ino_superosed}  
    \end{equation}
    where $m_{1x}^!$ is the intensity function of the reduced Palm version  of $\Phi_1$ at $x$. Besides,  the  $G$-function equals
    \begin{equation} \label{eq:G_ino_superposed}
        \begin{split}
             G^{in}_{\Phi}(r,x)  & = 1- \frac{m_1 (x)}{m_1 (x) + \rho_2}  \E \left[  \prod_{X \in \Phi^{!}_{1x}} \left(1- \frac{\bar m}{m_1 (X)+ \rho_2} \indicator_{B (x, r)}(X)  \right)\right] e^{- \rho_2 S (x, r)}\\
             & \qquad -
          \frac{\rho_2}{m_1 (x) + \rho_2}  \E \left[  \prod_{X \in \Phi_1} \left(1- \frac{\bar m}{m_1 (X)+ \rho_2} \indicator_{B (x, r)}(X)  \right)\right] e^{- \rho_2 S (x, r)},   
        \end{split}
    \end{equation}
    where $\bar m := \inf_{x \in \R^2}  m(x) >0$ and $S (x, r):= \int_{B (x, r)}  \bar m / (m_1 (u)+ \rho_2) \dd u$.
\end{proposition}
\begin{proof} 
First, we focus on the proof of \eqref{eq:K_ino_superosed}.
    Since $\Phi=\Phi_1 + \Phi_2$ and by exploiting the definition \eqref{eq:ino_K_def}, we get
    \begin{equation*}
        K_{\Phi}^{in} (r,x)  =  K_{\Phi_1+\Phi_2}^{in} (r,x) = \E \left[ \int_{B (x, r) } \frac{1}{m_{1} (u)+ \rho_2 } (\Phi_1+\Phi_2)^{!}_x (\dd u)\right].
    \end{equation*}
    An application of Theorem \ref{teo:main} for the Palm versions yields
    \begin{align*}
        K_{\Phi}^{in} (r,x) & = \frac{m_{1} (x) }{m_{1} (x)+ \rho_2}\E \left[ \int_{B (x, r) } \frac{1}{m_{1} (u)+ \rho_2 } (\Phi_{1x}^{!} + \Phi_2) (\dd u)\right]\\
        & \qquad+
        \frac{\rho_2 }{m_{1} (x)+ \rho_2}\E \left[ \int_{B (x, r) } \frac{1}{m_{1} (u)+ \rho_2 } (\Phi_1 + \Phi_{2x}^{!}) (\dd u)\right]\\
        & = \frac{m_{1} (x) }{m_{1} (x)+ \rho_2}\E \left[ \int_{B (x, r) } \frac{1}{m_{1} (u)+ \rho_2 } (\Phi_{1x}^{!} + \Phi_2) (\dd u)\right]\\
        & \qquad+
        \frac{\rho_2 }{m_{1} (x)+ \rho_2}\E \left[ \int_{B (x, r) } \frac{1}{m_{1} (u)+ \rho_2 } (\Phi_1 + \Phi_{2}) (\dd u)\right]\\
        & = \frac{m_{1} (x) }{m_{1} (x)+ \rho_2} \left\{ \E \left[ \int_{B (x, r) } \frac{1}{m_{1} (u)+ \rho_2 } \Phi_{1x}^{!}  (\dd u)\right] + \E \left[ \int_{B (x, r) } \frac{1}{m_{1} (u)+ \rho_2 } \Phi_2 (\dd u)\right] \right\}\\
        & \qquad+
        \frac{\rho_2 }{m_{1} (x)+ \rho_2}\left\{ \E \left[ \int_{B (x, r) } \frac{1}{m_{1} (u)+ \rho_2 } \Phi_1 (\dd u)\right] + \E \left[ \int_{B (x, r) } \frac{1}{m_{1} (u)+ \rho_2 } \Phi_{2} (\dd u)\right]  \right\} \\
    \end{align*}
    where we also observed that $\Phi_2 = \Phi_{2x}^!$ since $\Phi_2$ is a Poisson process. All the expected values in the last equation can be computed by virtue of the Campbell theorem to get
      \begin{align*}
        K_{\Phi}^{in} (r,x) & =\frac{m_{1} (x) }{m_{1} (x)+ \rho_2} \int_{B(x, r)} \frac{m_{1x}^! (u)+ \rho_2 }{m_{1}(u)+ \rho_2} \dd u  +
        \frac{\rho_2}{m_1 (x) + \rho_2} |B (x, r)|.
        \end{align*}
    By adding and subtracting the intensity function $m_1 (u)$ in the numerator of the first integral, we obtain
    \begin{align*}
        K_{\Phi}^{in} (r,x) & =\frac{m_{1} (x) }{m_{1} (x)+ \rho_2} \int_{B(x, r)} \frac{m_{1x}^! (u) - m_1 (u) }{m_{1}(u)+ \rho_2} \dd u  + |B (x, r)|
        \end{align*}
        and formula \eqref{eq:K_ino_superosed} follows by observing that $|B (x, r)| = \pi r^2$.\\

        As for the inhomogeneous $G$-function, we can proceed by similar arguments. Indeed, an application of Theorem \ref{teo:main} for Palm versions yields 
        \begin{align*}
            G^{in}_{\Phi}(r,x)  & = 1- \frac{m_1 (x)}{m_1 (x)+ \rho_2}\E \left[  \prod_{X \in \Phi^{!}_ {1x}+ \Phi_2} \left(1- \frac{\bar m}{m_1(X)+ \rho_2} \indicator_{B (x, r)}(X)  \right)\right]\\
            & \qquad - \frac{\rho_2 }{m_1 (x) + \rho_2}\E \left[  \prod_{X \in \Phi_1+ \Phi^{!}_{2x}} \left(1- \frac{\bar m}{m_1 (X)+ \rho_2} \indicator_{B (x, r)}(X)  \right)\right] .
        \end{align*}
Since $\Phi_2$ is a Poisson process and by the independence assumption, we can write:
\begin{align*}
            G^{in}_{\Phi}(r,x)  & = 1- \frac{m_1 (x)}{m_1 (x)+ \rho_2}\E \left[  \prod_{X \in \Phi^{!}_ {1x}+ \Phi_2} \left(1- \frac{\bar m}{m_1(X)+ \rho_2} \indicator_{B (x, r)}(X)  \right)\right] \nonumber\\
            & \qquad - \frac{\rho_2 }{m_1 (x) + \rho_2}\E \left[  \prod_{X \in \Phi_1 +\Phi_2} \left(1- \frac{\bar m}{m_1 (X)+ \rho_2} \indicator_{B (x, r)}(X)  \right)\right] \nonumber \\
            & = 1- \frac{m_1 (x)}{m_1 (x)+ \rho_2}\E \left[  \prod_{X \in \Phi^{!}_ {1x}} \left(1- \frac{\bar m}{m_1(X)+ \rho_2} \indicator_{B (x, r)}(X)  \right)\right] \E \left[  \prod_{X \in \Phi_2} \left(1- \frac{\bar m}{m_1(X)+ \rho_2} \indicator_{B (x, r)}(X)  \right)\right] \nonumber\\
            & \qquad - \frac{\rho_2 }{m_1 (x) + \rho_2}\E \left[  \prod_{X \in \Phi_1 } \left(1- \frac{\bar m}{m_1 (X)+ \rho_2} \indicator_{B (x, r)}(X)  \right)\right] \E \left[  \prod_{X \in\Phi_2} \left(1- \frac{\bar m}{m_1 (X)+ \rho_2} \indicator_{B (x, r)}(X)  \right)\right].
        \end{align*}
        Observe that
        \begin{equation} \label{eq::G_ino_proof}
        \begin{split}
          & \E \left[  \prod_{X \in \Phi_2} \left(1- \frac{\bar m}{m_1 (X)+ \rho_2} \indicator_{B (x, r)}(X)  \right)\right] \\
          & \qquad= \E \exp \left\{ \sum_{X \in \Phi_2 }\log \left( 1- \frac{\bar m}{m_1 (X)+ \rho_2} \indicator_{B (x, r)}(X)  \right)\right\}\\
            & \qquad= \exp \left\{ \int_{\R^2} \frac{\bar m}{m_1 (u)+ \rho_2} \indicator_{B (x, r)} (u) \rho_2 \dd u\right\}\\
             & \qquad = \exp \left\{ \rho_2 \int_{B (x, r)} \frac{\bar m}{m_1 (u)+ \rho_2} \dd u\right\} = 
             e^{- \rho_2 S (x, r)},
        \end{split}
        \end{equation}
        where we have used the explicit expression for the Laplace functional of a Poisson  point process.
        By substituting \eqref{eq::G_ino_proof} in the last expression for $G^{in}_{\Phi}(r,x)$, we get the second formula of the proposition.
\end{proof}

We now focus on a simple example below. 
\begin{example}[Corrupted determinantal point process]
Here we describe an application where $\Phi_1$ is a determinantal point process (\textsc{dpp}). We remind that the class of \textsc{dpp}s  are a class of repulsive point processes, see \citet{Lavancier2015,KuleszaTaskar2012} and references therein for a wide range of statistical application of \textsc{dpp}s. More precisely, let $R \subset \mathbb R^2$ be a compact set, a \textsc{dpp} $\Phi_1$ on $R$ is specified by a covariance kernel $C: R \times R \rightarrow \mathbb C$, such that the $k$-th factorial moment measure $M_{\Phi_1}^{(k)}$ equals
\[
    M_{\Phi_1}^{(k)}(\dd x_1\, \ldots\, x_k) = \det\{C(x_h, x_w)_{h,w = 1,\ldots,k}\} \dd x_1\,\ldots\, \dd x_k, \qquad x_1, \ldots, x_k \in R,
\]
where $C(x_h, x_w)_{h,w = 1,\ldots,k}$ is the $k \times k$ matrix with entries $C(x_h, x_w)$. Thus, in this case the intensity function $m_1 (x)$ equals $C (x,x)$.
We also remind that the reduced Palm version of a \textsc{dpp} is again a \textsc{dpp} with a new covariance kernel given by
\[
C^{!}_x (u, v) = C (u, v) - \frac{C (u, x) C(v, x)}{C (x,x)}.
\]
Hence,  the intensity function of the reduced Palm \textsc{dpp} equals
\[
m_{1x}^! (u) =  C^{!}_x (u, u) = C (u, u) - \frac{C^2 (u, x)}{C (x,x)} . 
\]

We first concentrate on the evaluation of the inhomogeneous $K$-function \eqref{eq:K_ino_superosed}.
By substituting the expressions of the intensity functions for the \textsc{dpp} and its reduced Palm version  in  \eqref{eq:K_ino_superosed}, the inhomogeneous $K$-function of the superposition $\Phi_1+\Phi_2$ boils down to
\begin{equation*}
     K_{\Phi}^{in} (r,x)  = \pi r^2  - \int_{B (x,r)} \frac{C^2 (u,x)}{(C(x,x)+ \rho_2) \cdot(C(u,u)+ \rho_2)} \dd u .
\end{equation*}

Second, to derive an expression of the inhomogeneous $G$-function in \eqref{eq:G_ino_superposed}, we have to evaluate the two expected values. In order to do this, we remind that when $\Phi_1$ is a \textsc{dpp} and $\varphi : R \to [0,1]$, we have:
\begin{equation} \label{eq:prod_PHI_1}
    \E \left[ \prod_{X \in \Phi_1} \varphi (X)\right] = \det \left(  \mathcal{I} - \mathcal{K}_v \right) 
\end{equation}
where $\det$ denotes the Fredholm determinant of the difference between the identity operator $\mathcal{I}$ and the integral operator $\mathcal{K}_\varphi$ with kernel 
\[
C_v (u,v) = \sqrt{1-\varphi (u)} C (u,v) \sqrt{1- \varphi (v)}.
\]
By specializing \eqref{eq:prod_PHI_1} for the second expected value in \eqref{eq:G_ino_superposed}
, we get
\begin{equation}
    \label{eq:second_product_G_ino}
   \E \left[  \prod_{X \in \Phi_1} \left(1- \frac{\bar m}{m_1 (X)+ \rho_2} \indicator_{B (x, r)}(X)  \right)\right]  =
   \det \left( \mathcal{I} - \mathcal{K} \right)
\end{equation}
where $\mathcal{K}$ is the integral operator having kernel
\[
K (u,v) = \sqrt{\frac{\bar m}{C (u,u)+ \rho_2} \indicator_{B (x,r)} (u)} \cdot  C (u,v) \cdot
\sqrt{\frac{\bar m}{C (v,v)+ \rho_2} \indicator_{B (x,r)} (v)}.
\]
Since $\Phi_{1x}^!$ is another \textsc{dpp}, a similar argument shows that the first expected value in \eqref{eq:G_ino_superposed} equals
\begin{equation}
    \label{eq:first_product_G_ino}
   \E \left[  \prod_{X \in \Phi_{1x}^!} \left(1- \frac{\bar m}{m_1 (X)+ \rho_2} \indicator_{B (x, r)}(X)  \right)\right]  =
   \det \left( \mathcal{I} - \mathcal{K}_{x}^! \right)
\end{equation}
where $\mathcal{K}_{x}^!$ is the integral operator having kernel
\[
K_{x}^! (u,v) = \sqrt{\frac{\bar m}{C (u,u)+ \rho_2} \indicator_{B (x,r)} (u)} \cdot  C^{!}_x (u,v) \cdot
\sqrt{\frac{\bar m}{C (v,v)+ \rho_2} \indicator_{B (x,r)} (v)}.
\]
By substituting \eqref{eq:second_product_G_ino} and \eqref{eq:first_product_G_ino} in the expression  of the 
inhomogeneous $G$-function \eqref{eq:G_ino_superposed}, we get:
    \begin{equation*}
             G^{in}_{\Phi}(r,x)   = 1- \frac{m_1 (x)}{m_1 (x) + \rho_2}   \det \left( \mathcal{I} - \mathcal{K}_{x}^! \right)e^{- \rho_2 S (x, r)}
          \frac{\rho_2}{m_1 (x) + \rho_2}  \det \left( \mathcal{I} - \mathcal{K} \right) e^{- \rho_2 S (x, r)}.
    \end{equation*}

\end{example}

\subsection{Summary statistics for stationary processes} \label{app:stat_summary}

We now discuss the use of Theorem \ref{teo:main} to derive expressions of summary statistics under the condition of stationarity for $\Phi= \Phi_1+ \Phi_2$.\\

\noindent
\textbf{$K$-function.} The Ripley's $K$-function \citep{ripley1976second}, for a stationary point process $\Phi$ with intensity $M_\Phi \equiv \rho > 0$,  is defined as 
\[
    K_{\Phi}(r) = \frac{1}{\rho} \E[\Phi^!_{o}(B(o,r))], \quad \text{for } r \geq 0,
\]
where $o$ denotes a generic point of $\R^2$ and $B(o, r)$ is the ball of radius $r$ centered at $o$.
Thanks to stationarity, $K_{\Phi}$ is invariant to the choice of point $o$, which is called the \emph{typical point} of $\Phi$: the quantity $\rho K_{\Phi}(r)$ represents the expected number of points that are $r$-close to a generic point $o$, given that $\Phi$ has an atom in $o$.

An application of Theorem \ref{teo:main} yields the $K$-function for the superposition of two independent and stationary point processes $\Phi_i$, $i=1,2$ with intensities $M_{\Phi_i} \equiv \rho_i$. Specifically, letting $\rho = \rho_1 + \rho_2$, 
\begin{equation}\label{eq:Ripley_superposition}
     K_{\Phi_1 + \Phi_2}(r) = \frac{1}{\rho} \left[ \left\{K_{\Phi_1}(r) \rho_1 + \rho_2 |B(o,r)| \right\} \frac{\rho_1}{\rho} +  \left\{ \rho_1 |B(o,r)| + K_{\Phi_2}(r) \rho_2\right\} \frac{\rho_2}{\rho}  \right] ,
\end{equation}
where $|B(o,r)|= \pi r^2$ in $\R^2$.\\

\noindent
\textbf{$G$-function.} The nearest neighborhood function (or $G$-function) for a stationary point process $\Phi$ is the cumulative distribution function of the distance between a typical point  $o \in \Phi$ to the nearest other point of $\Phi$ \citep{BaddeleyRubakTurner2015}. More specifically:
\[
G_\Phi (r) = P (\Phi_o^! (B (o, r))>0),
\]
and thanks to the stationarity, the function does not depend on the choice of $o$.
As for the superposed point process $\Phi= \Phi_1 + \Phi_2$, Theorem \ref{teo:main} leads us to the following
\[
\begin{split}
G_\Phi (r) &= 1- P (\Phi_o^! (B (o, r))=0) \\
& =
1- \left[\frac{\rho_1}{\rho_1 + \rho_2} (1-G_{\Phi_1} (r)) (1- F_{\Phi_2} (r))+ \frac{\rho_2}{\rho_1 + \rho_2}
(1- F_{\Phi_1} (r)) (1-F_{\Phi_2} (r))
\right]
\end{split}
\]
where $F_{\Phi_i } (r) =  P (\Phi_i (B (o, r))>0)$ is the contact distribution function of $\Phi_i$. Since $\Phi_2$ is a Poisson process, we have that $(1-F_{\Phi_2} (r))= e^{- \pi r^2 \rho_2}$, then:
\begin{equation}
    \label{eq:G_function_superposed}
    G_\Phi (r) =
1- \left[\frac{\rho_1}{\rho_1 + \rho_2} (1-G_{\Phi_1} (r)) e^{- \pi r^2 \rho_2}+ \frac{\rho_2}{\rho_1 + \rho_2}
(1- F_{\Phi_1} (r)) e^{- \pi r^2 \rho_2}
\right].
\end{equation}

\noindent
\textbf{$A$-function.}
We finally concentrate on the $A$-function, introduced by \cite{chiu2008reduced}. This is defined as
\[
A_{\Phi} (s, r) = \E \left( s^{\Phi_o^! (B (o, r))}\right)
\]
where $r \geq 0$ and $|s|<1$. For fixed $r$, this quantity represents the probability generating function of 
$\Phi_o^! (B (o, r))$. If the point process $\Phi$ is stationary, then the $A$-function does not depend on the choice of the point $o$, and so we can consider the origin as a reference point. A plain application of Theorem \ref{teo:main} and the fact that $\Phi_2$ is a Poisson point process lead us to the following formula for the $A$-function:
\[
 A_{\Phi} (s, r) = \frac{\rho_1}{\rho_1+ \rho_2}   A_{\Phi_1} (s, r) \mathcal{G}_{\Phi_2 (B (o, r))} (s)
  +  \frac{\rho_2}{\rho_1+ \rho_2}   \mathcal{G}_{\Phi_1 (B (o, r))} (s) \mathcal{G}_{\Phi_2 (B (o, r))} (s)
\]
where $\mathcal{G}_X (s)$ denotes the probability generating function of a random variable $X$ evaluated at $s$.
In addition, by exploiting the expression of $\mathcal{G}_{\Phi_2 (B (o, r))} (s)$ for a Poisson random variable, we can write
\begin{equation}
    \label{eq:A_superposed}
    A_{\Phi} (s, r) = \frac{\rho_1}{\rho_1+ \rho_2}  e^{\rho_2 \pi r^2 (s-1)} A_{\Phi_1} (s, r)
  +  \frac{\rho_2}{\rho_1+ \rho_2}  e^{\rho_2 \pi r^2 (s-1)} \mathcal{G}_{\Phi_1 (B (o, r))} (s).
\end{equation}

\section{Minimum contract estimation for stationary point processes} \label{app:MCE}

In Section \ref{app:details_matern} we give the details for the evaluation of the $A$-function and the $K$-function used in the experiment of Section \ref{sec:corrupted_Matern} to fit the corrupted Mat\'ern cluster process via \textsc{mce}. In Section \ref{app:corrupted_dpp}, we provide another application of \textsc{mce} to fit a corrupted point process $\Phi=\Phi_1 + \Phi_2$, where $\Phi_1$ is a repulsive point pattern, i.e., a determinantal point process, and $\Phi_2$ is a random background noise. The class of determinantal point processes is briefly recalled in Section \ref{app:corrupted_dpp}.

\subsection{Fitting a corrupted Mat\'ern cluster process: details for Section \ref{sec:corrupted_Matern}}
\label{app:details_matern}

Here we provide the explicit expressions of the  $A$-function and the $K$-function used to fit the corrupted Mat\'ern cluster process via minimum contrast estimation
in Section \ref{sec:corrupted_Matern}.
Having set
\[
\kappa ( \xi ; c ) = \frac{1}{\pi R^2} \indicator_{B (c, R)} (\xi),
\]
with $R >0$, recall that a Mat\'ern cluster process is defined as
\[
\Phi_1 \mid \Phi_p \sim \textsc{pp} \left( \int_{\R^2}  \mu  \kappa (\xi ; c) \Phi_p (\dd c) \dd \xi \right), \quad 
\Phi_p  \sim \textsc{pp} (\lambda \dd c)
\]
where $\textsc{pp} (\lambda \dd c)$ denotes a homogeneous Poisson process with intensity $\lambda$ on $\R^2$, and $\lambda, \mu >0$.
The Mat\'ern cluster process is a particular instance  of the shot noise Cox process \citep{Mo03Cox}, thus the Palm distribution at a single point $x \in \R^2$ can be found in \citep{Mo03Cox}. We have
\begin{equation}
    \label{eq:shot_single}
    \Phi_{1x}^! \stackrel{d}{=} \Phi_1 + \Phi_{\zeta_x},
\end{equation}
where:
\begin{itemize}
    \item[(i)]  $\Phi_1$ is a Mat\'ern cluster process, distributed as the original one;
    \item[(ii)]  $\Phi_{\zeta_x} $ is a single cluster of points, where 
    \[
    \Phi_{\zeta_x} \mid \zeta_x \sim \textsc{pp} \left(  \mu \kappa ( \xi ; \zeta_x) \dd \xi \right),
    \quad
    \zeta_x \sim f_{\zeta_x} (c) \propto \kappa (x ; c).
    \]
\end{itemize}
Finally,  $\Phi_1$ and $\Phi_{\zeta_x}$ on the right hand side of \eqref{eq:shot_single} are independent. It can be shown that the Mat\'ern cluster process $\Phi_1$ is stationary with intensity $\rho_1 = \lambda\mu$.\\

Now consider the corrupted point process $\Phi= \Phi_1+ \Phi_2$, where $\Phi_1$ is a Mat\'ern cluster process and $\Phi_2$ is a stationary Poisson process with intensity $\rho_2$.
We first focus on the evaluation of the $A$-function, whose expression equals \eqref{eq:A_superposed}.
Since the process is stationary, the $A$-function does not depend on the specific point $x$, so we can use the origin $o$ as a reference point. 
Thanks to the representation of the Palm distribution \eqref{eq:shot_single}, the $A$-function of $\Phi_1$ equals
\begin{equation} \label{eq:A_1_shot_noise}
\begin{split}
      A_{\Phi_1} (s, r) & = \E \left[ s^{\Phi_{1o}^! (B(o,r))}\right] = \E \left[ s^{ \Phi_1 (B (o, r))}\right] \E \left[s^{ \Phi_{\zeta_o} (B (o, r))}\right] \\
   & =  \mathcal{G}_{\Phi_1 (B (o, r))} (s) \E \left[s^{ \Phi_{\zeta_o} (B (o, r))}\right] .
\end{split}
\end{equation}
The expected value in \eqref{eq:A_1_shot_noise}, can be evaluated as follows:
\begin{align*}
    \E \left[s^{ \Phi_{\zeta_o} (B (o, r))}\right] & = \E \E  \left[s^{ \Phi_{\zeta_o} (B (o, r))} \mid \zeta_o\right]  = \E \left[ e^{\mu |B (o, r) \cap B (\zeta_o, R)| (s-1)/(\pi R^2)}\right]
\end{align*}
where we have used the fact that the inner expectation is the probability generating function of a Poisson random variable. Now, we can integrate out $\zeta_o$, whose density is $\kappa(o ; c)$:
\begin{align*}
    \E \left[s^{ \Phi_{\zeta_o} (B (o, r))}\right] & = \int_{B (o, R)} e^{\mu |B (o, r) \cap B (c, R)| (s-1)/(\pi R^2)} \frac{1}{\pi R^2} \dd c.
\end{align*}
Therefore \eqref{eq:A_1_shot_noise} boils down to
\[
A_{\Phi_1} (s, r)  =   \mathcal{G}_{\Phi_1 (B (o, r))} (s) \int_{B (o, R)} e^{\mu |B (o, r) \cap B (c, R)| (s-1)/(\pi R^2)} \frac{1}{\pi R^2} \dd c .
\]
We can plug in the previous expression in \eqref{eq:A_superposed} to obtain the $A$-function of the superposed point process $\Phi$:
\begin{equation}
    \label{eq:A_superposed_shot_noise}
    \begin{split}
    A_{\Phi} (s, r) &= \left[\frac{\rho_1}{\rho_1+ \rho_2} \int_{B (o, R)}  e^{\mu |B (o, r) \cap B (c, R)| (s-1)/(\pi R^2)} \frac{1}{\pi R^2} \dd c
\right.  \\
& \qquad \qquad \left. +  \frac{\rho_2}{\rho_1+ \rho_2}  \right] e^{\rho_2 \pi r^2 (s-1)}\mathcal{G}_{\Phi_1 (B (o, r))} (s),
\end{split}
\end{equation}
where it remains to evaluate the probability generating function of the Mat\'ern cluster process $\Phi_1$. In order to do this, observe that the following representation holds true
\[
\Phi_1 \mid \Phi_p \stackrel{d}{=} \sum_{i \geq 1} \Phi_{1i} , \quad \Phi_{1i }  \mid \Phi_p \simind \textsc{pp} (\mu \kappa (\xi ; c_i) \dd \xi )
\]
where $\Phi_p = \sum_{i \geq 1} \delta_{c_i}$ is a homogeneous Poisson process with intensity $\lambda$. Thus, the previous representation entails
\begin{align*}
    \mathcal{G}_{\Phi_1 (B (o, r))} (s) & = \E \E \left( s^{\Phi_1 (B (o, r))}\mid \Phi_p\right) = \E \left[ \prod_{i \geq 1}  \E \left( s^{\Phi_{1i} (B (o, r))}\mid \Phi_p\right)\right] \\
    & = \E \left[ \prod_{i \geq 1}   \exp \left\{ \frac{\mu}{\pi R^2} |B (o, r) \cap B (c_i, R)| (s-1) \right\}\right] \\
        & = \E \left[  \exp \left\{ \int_{\R^2}\frac{\mu}{\pi R^2} |B (o, r) \cap B (c, R)| (s-1) \Phi_p (\dd c) \right\}\right] \\
        & = \exp \left\{ - \int_{\R^2} \left[ 1- \exp \left\{ \frac{\mu}{\pi R^2} |B (o, r) \cap B (c, R)|  (s-1)\right\}\right] \lambda \dd c  \right\}.
\end{align*}
The integral over $\R^2$ can be evaluated by a change of variables (polar coordinates) to get
\begin{align*}
    \mathcal{G}_{\Phi_1 (B (o, r))} (s) &  = \exp \left\{ - 2 \pi \int_0^{R+r} \left[ 1- \exp \left\{ \frac{\mu}{\pi R^2} I_{R,r} (\varrho)  (s-1)\right\}\right] \lambda \varrho \dd \varrho  \right\},
\end{align*}
 having defined $I_{R,r} (\varrho) := |B (o, r) \cap B ((\varrho, 0), R)|$. As a consequence, \eqref{eq:A_superposed_shot_noise} becomes
\begin{equation}
    \label{eq:A_superposed_shot_noise_final}
    \begin{split}
    &A_{\Phi} (s, r) = \left[\frac{\rho_1}{\rho_1+ \rho_2}\int_0^R  \exp \left\{\frac{\mu}{\pi R^2} I_{R, r} (\varrho) (s-1)\right\}  \frac{2 \varrho }{R^2} \dd \varrho
 +  \frac{\rho_2}{\rho_1+ \rho_2}  \right]  \\
& \qquad \times e^{\rho_2 \pi r^2 (s-1)}\exp \left\{ - 2 \pi \int_0^{R+r} \left[ 1- \exp \left\{ \frac{\mu}{\pi R^2} I_{R,r} (\varrho)  (s-1)\right\}\right] \lambda \varrho \dd \varrho  \right\},
\end{split}
\end{equation}
where
\[
I_{R,r} (\varrho) = \left\{ \begin{array}{cc}
   \pi \min \{ r, R\}^2  & \text{if } \varrho \leq |r-R|  \\
   r^2 \arccos{\frac{\varrho^2 + r^2 -R^2}{2 \varrho R}} + R^2 \arccos{ \frac{\varrho^2+ R^2 -r^2 }{2 \varrho R}} - \frac{1}{2} 
   \sqrt{p (\varrho, r, R)}& \text{otherwise }
\end{array}
\right.
\]
having set $p (\varrho, r, R) : = (-\varrho + r + R) (\varrho + r - R) (\varrho-r + R) (\varrho + R + r)$.\\


As for the $K$-function, we can exploit \eqref{eq:Ripley_superposition} to get
\begin{equation} \label{eq:Matern_1K}
     K_{\Phi_1 + \Phi_2}(r) = \frac{1}{\rho} \left[ \left\{K_{\Phi_1}(r) \rho_1 + \rho_2 \pi r^2\right\} \frac{\rho_1}{\rho} +  \left\{ \rho_1 \pi r^2 + \pi r^2\rho_2\right\} \frac{\rho_2}{\rho}  \right] ,
\end{equation}
where we used the fact that $K_{\Phi_2}(r) = \pi r^2 $, since $\Phi_2$ is a Poisson process, and that $|B(o,r)|  = \pi r^2$. In addition,  $ K_{\Phi_1}(r)$ can be evaluated by resorting to \eqref{eq:shot_single} to obtain
\[
K_{\Phi_1}(r)
=
\pi r^2
+
\frac{2\pi}{\lambda \pi^2 R^4}
\int_0^{R}
\varrho I_{R,r} (\varrho) \dd \varrho
\]
with some simple calculations. Putting the previous expression in  \eqref{eq:Matern_1K}, we obtain 
\[
\begin{split}
     K_{\Phi_1 + \Phi_2}(r) =  2\frac{\rho_1\rho_2 \pi r^2}{\rho^2} +
    \pi r^2\frac{\rho_1^2}{\rho^2} + \pi r^2 \frac{\rho_2^2}{\rho^2} + \frac{\rho_1^2}{\rho^2} \cdot\frac{2\pi}{\lambda \pi^2 R^4}
\int_0^{R}
\varrho I_{R,r} (\varrho) \dd \varrho  .
\end{split}
\]

\subsection{Fitting a corrupted determinantal point process} \label{app:corrupted_dpp}

Determinantal point processes (\textsc{dpp}s)  are a class of repulsive point processes, see \citet{Lavancier2015,KuleszaTaskar2012} and references therein for a wide range of statistical application of \textsc{dpp}s.

Let $R \subset \mathbb R^2$ be a compact set. 
A \textsc{dpp} $\xi$ on $R$ is specified by a covariance kernel $C: R \times R \rightarrow \mathbb C$, such that the $k$-th factorial moment measure $M_{\xi}^{(k)}$ equals
\[
    M_{\xi}^{(k)}(\dd x_1\, \ldots\, x_k) = \det\{C(x_h, x_w)_{h,w = 1,\ldots,k}\} \dd x_1\,\ldots\, \dd x_k, \qquad x_1, \ldots, x_k \in R,
\]
where $C(x_h, x_w)_{h,w = 1,\ldots,k}$ is the $k \times k$ matrix with entries $C(x_h, x_w)$. In particular, we assume that $\xi$ follows a Gaussian \textsc{dpp} \citep{Lavancier2015} with kernel $C(x, y) = \rho_\xi \exp\{-\| (x - y) / \alpha \|^2 \}$ parametrized by $(\rho_\xi, \alpha)$, with $\rho_\xi < (\pi \alpha^2)^{-1}$ to ensure the process is well-defined.
From \cite{Lavancier2015}, Ripley's $K$-function of $\xi$ is
\begin{equation}\label{eq:dpp_k}
    K_{\xi}(r) = \pi r^2 - \frac{\pi \alpha^2}{2}\left(1 - \e^{-2r^2/\alpha^2}\right).
\end{equation}
While \textsc{dpp}s allow for likelihood-based inference as discussed in \cite{Lavancier2015}, this is typically numerically cumbersome due to a Fourier series expansion and the determinant of large matrices involved in the likelihood. 
Hence, in the \texttt{spatstat} package \citep{BaddeleyTurner2005}, the default way of fitting a \textsc{dpp} is through \textsc{mce} based on the $K$-function.

Assume now to observe a realization of $\xi$ corrupted by a background noise, independent of $\xi$. Let $\Phi_2$ be a homogeneous Poisson point process on $R$ with intensity $\omega$, which models the corrupting noise.
This setting naturally fits within the modeling framework of Section \ref{sec:corrupted_Matern}.
Plugging \eqref{eq:dpp_k} into \eqref{eq:Ripley_superposition}, and recalling that for the homogeneous Poisson process $\Phi_2$ we have $K_{\Phi_2}(r) = \pi r^2$, the $K$-function for $\Phi = \xi + \Phi_2$ equals
\begin{equation}\label{eq:noisy_K}
    K_{\xi + \Phi_2}(r) = \pi r^2 - \frac{\rho_\xi^2}{(\rho_\xi + \omega)^2} \frac{\pi \alpha^2}{2} \left(1 - \e^{-2r^2/\alpha^2}\right).
\end{equation}

Let $\underline{x} = (x_1, \ldots, x_k) \in R^k$ denote the observed point pattern. 
The goal is to estimate the parameters of $\Phi = \xi + \Phi_2$, namely the intensity $\rho_\xi$ of the signal process, the repulsion parameter $\alpha$, and the noise intensity $\omega$. As customary, we estimate the overall intensity $\rho = \rho_\xi + \omega$ of $\Phi$ by $\hat \rho = k / |R|$, so that we are left with estimating only $\rho_\xi$ and $\alpha$.
Following the discussion above, we apply \textsc{mce} based on the minimization of the discrepancy between the $K$ function in \eqref{eq:noisy_K} and  the edge-corrected nonparametric estimator discussed by \cite{ripley1976second} (see also \citet{MoWaBook03}).
For numerical purposes, we approximate the integral via numerical quadrature using Simpson's rule. The minimization is performed via the \textsc{bfgs} algorithm using the \texttt{Julia} programming language.


\begin{table}
	\centering
\begin{tabular}{ccccccc}
    \multicolumn{3}{c}{True parameters} & \multicolumn{2}{c}{Gaussian \textsc{dpp}} & \multicolumn{2}{c}{Gaussian \textsc{dpp} + Poisson noise} \\
    $\rho_\xi$ & $\alpha$ & $u$ & $\hat{ \rho_\xi}$ & $\hat{\alpha}$ & $\hat{ \rho_\xi}$ & $\hat{\alpha}$ \\
    \midrule
    50 & 0.06  & 0.2  & 60.59 (6.99)  & 0.05 (0.02)  & 53.74 (11.56) & 0.06 (0.03) \\
    50 & 0.06  & 0.35 & 68.20 (7.82)  & 0.04 (0.02)  & 56.72 (14.20) & 0.06 (0.03) \\
    50 & 0.02  & 0.2  & 60.10 (8.19)  & 0.02 (0.02)  & 46.48 (16.69) & 0.03 (0.04) \\
    50 & 0.02  & 0.35 & 67.64 (9.20)  & 0.02 (0.02)  & 51.49 (19.40) & 0.03 (0.04) \\
   100 & 0.05  & 0.2  & 120.61 (9.35) & 0.04 (0.01)  & 107.25 (15.73)& 0.05 (0.01) \\
   100 & 0.05  & 0.35 & 135.71 (10.53)& 0.03 (0.01)  & 111.83 (21.14)& 0.05 (0.01) \\
   100 & 0.025 & 0.2  & 120.18 (11.15)& 0.02 (0.01)  & 98.30 (28.35) & 0.03 (0.03) \\
   100 & 0.025 & 0.35 & 135.22 (12.54)& 0.02 (0.01)  & 104.31 (33.28)& 0.03 (0.03) \\
  \end{tabular}
  \caption{Mean and standard deviation (in brackets) of the estimates  $(\hat{\rho_\xi}, \hat{\alpha})$ of the Gaussian \textsc{dpp} over $1,000$ independent replicated datasets, for different combinations of the true parameters reported in the left column. Middle column: parameter estimates when fitting only the Gaussian \textsc{dpp}. Right column: parameter estimates when fitting the superposition of the Gaussian \textsc{dpp} and the background noise Poisson process. }
	\label{tab:exp1}
\end{table}

We show an illustrative simulation to assess the performance of the \textsc{mce} approach based on the Ripley's $K$-function in estimating the parameters of $\Phi$. 
We generate a point pattern from a Gaussian \textsc{dpp} on the unit-square $R$ with parameters $(\rho_\xi, \alpha) \in \{(50, 0.06), (50, 0.02), (100, 0.05), (100, 0.025)\}$, and perturb the point pattern adding a realization from a homogeneous Poisson process with intensity $\omega = u \rho_\xi$ for $u \in \{0.2, 0.35\}$. 
The two rightmost columns of Table \ref{tab:exp1} show the mean and standard deviation of the estimated values $(\hat{\rho_\xi}, \hat{\alpha})$ over $1,000$ independent replicated datasets, for each combination of the true parameters. 
For comparison (two middle columns), we fit Gaussian \textsc{dpp}s to the same (contaminated) point patterns, neglecting the presence of the background noise, via \textsc{mce} using Ripley's $K$-function as implemented in the \texttt{spatstat} package.
It is clear how ignoring the contamination leads to substantial bias in the estimates for $\rho_\xi$ and $\alpha$, which is mitigated by our approach. However, since our approach requires estimating the parameters of two processes instead of one, our estimators exhibit larger variances.

\section{Results and proofs for the shot noise Cox process of Section \ref{sec:sncp_inference}}

\subsection{Auxiliary results for the shot noise Cox process}\label{app:auxiliary_results}

The first lemma describes the Laplace functional of a \textsc{sncp}.

\begin{lemma}\label{lemma:laplace_sncp}
Let $\Phi \sim \textsc{sncp}(\kappa, \nu)$ and any measurable function $f: \X \to \R$. Then,
\[
\mathcal{L}_{\Phi}(f) = \exp\left\{ - \int_{\X\times\R_+} \left( 1- \exp\left\{ 
-\gamma \int_{\X} (1 - \exp\{-f(x)\} ) \kappa(x; \theta) \dd x \right\} \right) \nu(\dd \theta \dd \gamma) \right\}.
\]
\end{lemma}
\begin{proof}
    See Appendix \ref{app:auxiliary_proofs}.
\end{proof}

The second lemma describes the Laplace functional, the mean measure and the reduced Palm version of the processes $\Phi_{\zeta_{\underline{x}_h}}$ appearing in the characterization of the reduced Palm distributions of the \textsc{sncp} in \Cref{thm:red_palm_sncp}.
\begin{lemma}\label{lemma:laplace_cox_terms}
Let $\underline{x}= (x_1,\ldots,x_k)$ and $\Phi_{\zeta_{\underline{x}}}$ be such that
\[
\Phi_{\zeta_{{\underline{x}}}} \mid \zeta_{{\underline{x}}} = (\theta_{{\underline{x}}}, \gamma_{{\underline{x}}}) \sim \textsc{pp}(\gamma_{{\underline{x}}} \kappa(x; \theta_{{\underline{x}}}) \,\dd x), \quad \zeta_{{\underline{x}}} = (\theta_{{\underline{x}}}, \gamma_{{\underline{x}}}) \sim f_{{\underline{x}}}(\dd \theta\, \dd \gamma) \propto \gamma^{k} \prod_{j=1}^k \kappa(x_j; \theta) \nu(\dd \theta \, \dd \gamma).
\]
The following holds true.
\begin{itemize}
    \item[(i)] For any measurable $f:\X \to \R$, the Laplace functional of $\Phi_{\zeta_{\underline{x}}}$ is equal to 
\[
\mathcal{L}_{\Phi_{\zeta_{\underline{x}}}}(f) = \int_{\X\times\R_+} \exp\left\{ -\gamma \int_{\X} (1 - \exp\{- f(x)\}) \kappa(x; \theta) \dd x \right\} \gamma^k \prod_{j=1}^k \kappa(x_j;  \theta)\nu(\dd \theta\, \dd \gamma) / \eta(\underline{x}),
\]
where $\eta(\underline{x}) = \int_{\X\times\R_+} \gamma^k \prod_{j=1}^k \kappa(x_j; \theta) \nu(\dd \theta \,\dd \gamma)$.\\
\item[(ii)] The mean measure of $\Phi_{\zeta_{\underline{x}}}$ equals
\begin{equation}\label{eq:mean_measure_components}
M_{\Phi_{\zeta_{\underline{x}}}}(\dd y) = \eta(\underline{x}, \dd y)/\eta(\underline{x}).    
\end{equation}
\item[(iii)] The reduced Palm version of $\Phi_{\zeta_{\underline{x}}}$ at $y\in \X$ corresponds to
\begin{equation}\label{eq:red_Palm_components}
    \left( \Phi_{\zeta_{{\underline{x}}}} \right)^!_{y} \deq \Phi_{\zeta_{({\underline{x}}, y)}}
\end{equation}
\end{itemize}
\end{lemma}
\begin{proof}
    See Appendix \ref{app:auxiliary_proofs}.
\end{proof}

The third lemma describes the $k$-th order factorial moment measure of a \textsc{sncp}. To state this result in a rigorous way, let us introduce the space $\calT_k$ as follows. Let $\underline{t} = (t_1,\ldots,t_k)$ be a vector of natural numbers such that $\max_j t_j = |\underline{t}|$, where $|\underline{t}|$ denotes the number of distinct values in $\underline{t}$. Consider the partition $\calP_k = \{\calC_1,\ldots,\calC_{|\underline{t}|}\}$ of $\{1,\ldots,k\}$ induced by the ties in $\underline{t}$, such that $j \in \calC_h$ iff $t_j = h$. Then, the space $\calT_k$ contains the equivalence classes of $\underline{t}$ inducing the same partition $\calP_k$. This is the correct space where the vector of latent indicators $\underline{T}$ takes value in \Cref{thm:red_palm_sncp} and \Cref{thm:janossy_sncp}. 
\begin{lemma}\label{lemma:k_factorial_sncp}
Let $\Phi \sim \textsc{sncp}(\kappa, \nu)$. The $k$-th order factorial moment measure, denoted with $M^{(k)}_{\Phi}$, is equal to 
\[
M^{(k)}_{\Phi}(\dd \underline{x}) = \sum_{\underline{t} \in \calT_k} \prod_{h=1}^{|\underline{t}|} \eta(\underline{x}_h)\dd \underline{x},
\]
where $\underline{x}_h = (x_j: t_j = h)$, for $h=1,\ldots,|\underline{t}|$.
\end{lemma}
\begin{proof}
    See Appendix \ref{app:auxiliary_proofs}.
\end{proof}

The last theorem characterizes the joint distribution of the number of points $N$ of $\Phi$, when $\Phi$ is a \textsc{sncp}, together with the latent partition among its points, determined by the $\Phi_i$'s (see  Section \ref{sec:sncp_inference}). This result turns out to be useful for proving  \Cref{thm:red_palm_sncp}.

\begin{theorem}\label{prop:prior_partition_sncp}
Consider $\Phi \sim \textsc{sncp}(\kappa, \nu)$, for some $\kappa$ and $\int_{\X\times \R_+} (1-e^{-\gamma} ) \nu(\dd\theta \, \dd \gamma) < \infty$. The joint distribution on the number of points and the latent partition induced by $\Phi$ is
\[
\prob(N, \{\calC_1,\ldots, \calC_C\}) = \frac{1}{N!} e^{- \int_{\X\times\R_+}(1-e^{-\gamma}) \nu(\dd\theta \, \dd \gamma)} \prod_{h = 1}^C \int_{\X\times \R_+} e^{-\gamma} \gamma^{n_h} \nu(\dd\theta \, \dd \gamma), 
\]
where $n_h = | \calC_h |$, $N = \sum_{h=1}^C n_h$. Clearly, this law does not depend on the kernel $\kappa$.
\end{theorem}
\begin{proof}
    See Appendix \ref{app:auxiliary_proofs}.
\end{proof}

\subsection{Proof of \Cref{thm:red_palm_sncp}: the Palm distributions of shot noise Cox processes}

To prove \Cref{thm:red_palm_sncp} we proceed by induction. For $k=1$ the result is given in \cite{Mo03Cox} and for $k=2$ the result can be shown by direct calculation thanks to the Palm algebra. Suppose now the statements hold for a general $k$-tuple 
$\underline{x}_k = (x_1, \ldots, x_k)$ of distinct points, and let $x_{k+1} \neq x_j$, $j = 1, \ldots, k$. By the Palm algebra
\[
    \Phi^!_{\underline{x}_{k+1}} = (\Phi^!_{\underline{x}_k})_{x_{k+1}}.
\]
Observe that $\Phi^!_{\underline{x}_{k}}$ has a mixture representation (by the induction hypothesis). Let $\underline{t}_k = (t_1, \ldots, t_k) \in \calT_k$ be indicators describing a partition of $\underline{x}_k$ into $|\underline{t}|$ clusters. For ease of notation, define $\Psi_{\underline{t}_k} = \Phi + \sum_{h=1}^{|\underline{t}_k|} \Phi_{\zeta_{\underline{x}_h}}$ where the $\Phi_{\zeta_{\underline{x}_h}}$'s are as in the statement of the theorem. Then
\[
    \plaw_{\Phi^!_{\underline{x}_k}} = \sum_{\underline{t}_k \in \calT_k} \prob(\underline{T}_k = \underline{t}_k)  \plaw_{\Psi_{\underline{t}_k}}.
\]
An application of \Cref{prp:palm_mixture} yields
\[
    \plaw_{\Phi^!_{\underline{x}_{k+1}}} =  \sum_{\underline{t}_k \in \calT_k} w_{\underline{t}_k}(x_{k+1})  \plaw^{! x_{k+1}}_{\Psi_{\underline{t}_k}},
\]
where $\plaw^{!x_{k+1}}_{\Psi_{\underline{t}_k}}$ is the law of the process $\left(\Phi + \sum_{h=1}^{|\underline{t}_k|} \Phi_{\zeta_{\underline{x}_h}}\right)^!_{x_{k+1}}$ and, by \Cref{lemma:laplace_cox_terms},
\[
    w_{\underline{t}_k}(x_{k+1}) \propto \prob(\underline{T}_k = \underline{t}_k) \left(\eta(x_{k+1}) + \sum_{h=1}^{\underline{t}_k} \frac{\eta(\underline{x}_h, x_{k+1})}{\eta(\underline{x}_k)} \right) =: \prob(\underline{T}_k = \underline{t}_k) \lambda_{\underline{t}}(x_{k+1}) .
\]
Moreover
\[
  \left(\Phi + \sum_{h=1}^{|\underline{t}_k|} \Phi_{\zeta_{\underline{x}_h}}\right)^!_{x_{k+1}} \deq \begin{cases}
      \Phi + \Phi_{\zeta_{x_{k+1}}} + \sum_{h=1}^{|\underline{t}_k|} \Phi_{\zeta_{\underline{x}_h}} & \text{with prob. } \frac{\eta(x_{k+1})}{\lambda_{\underline{t}}(x_{k+1})} \\
      \Phi + \Phi_{\zeta_{(\underline{x}_j, x_{k+1})}} + \sum_{h \neq j} {|\underline{t}_k|} \Phi_{\zeta_{\underline{x}_h}}  & \text{with prob. } \frac{\eta(\underline{x}_h, x_{k+1}) / \eta (\underline{x}_h)}{\lambda_{\underline{t}}(x_{k+1})}
  \end{cases}  
\]
Putting things together, we obtain
\[
    \Phi^!_{\underline{x}_{k+1}} \deq 
    \begin{cases}
        \Phi + \Phi_{\zeta_{x_{k+1}}} + \sum_{h=1}^{|\underline{t}_k|} \Phi_{\zeta_{\underline{x}_h}} & \text{with prob. }  w_{\underline{t}_k}(x_{k+1}) \frac{\eta(x_{k+1})}{\lambda_{\underline{t}}(x_{k+1})}  \\
        \Phi + \Phi_{\zeta_{(\underline{x}_j, x_{k+1})}} + \sum_{h \neq j} {|\underline{t}_k|} \Phi_{\zeta_{\underline{x}_h}}  & \text{with prob. } w_{\underline{t}_k}(x_{k+1})  \frac{\eta(\underline{x}_h, x_{k+1}) / \eta (\underline{x}_h)}{\lambda_{\underline{t}}(x_{k+1})}.
    \end{cases}
\]
The proof follows by observing that $w_{\underline{t}_k}(x_{k+1}) \frac{\eta(x_{k+1})}{\lambda_{\underline{t}}(x_{k+1})} \propto \prod_{h=1}^{|\underline{t}_k|} \eta(\underline{x}_h) \eta(x_{k+1})$
and $w_{\underline{t}_k}(x_{k+1})  \frac{\eta(\underline{x}_h, x_{k+1}) / \eta (\underline{x}_h)}{\lambda_{\underline{t}}(x_{k+1})} \propto \prod_{j \neq h}^{|\underline{t}|} \eta(\underline{x}_j) \eta(\underline{x}_h, x_{k+1})$.

\subsection{General statement and proof of \Cref{thm:janossy_sncp}: the Janossy measures of shot noise Cox processes}\label{app:general_janossy_sncp}

We provide the Janossy measures for finite \textsc{sncp}s, for any arbitrary intensity measure $\nu(\dd \theta\, \dd \gamma)$. \Cref{thm:janossy_sncp} corresponds to the special case with the additional assumption $\nu(\dd \theta \, \dd \gamma) = \rho(\dd \gamma) G_0(\dd \theta)$. This special case is discussed in the following statement as well. 
\begin{theorem}
    Let $\Phi \sim \textsc{sncp}(\kappa, \nu)$ such that $\Phi(\X) < \infty$ almost surely. Then, its Janossy density equals
    \begin{equation*}
        \begin{aligned}
            j_k(x_1,\ldots,x_k) = \exp\left\{ - \int_{\X\times \R_+} ( 1 - e^{-\gamma}) \nu(\dd \theta\, \dd \gamma) \right\} \sum_{\underline{t} \in \calT_k} \prod_{h=1}^{|\underline{t}|} \int_{\X\times \R_+} e^{-\gamma} \gamma^{n_h} \prod_{j: t_j = h} \kappa(x_j ; \theta) \nu(\dd \theta \,\dd \gamma).
        \end{aligned}
    \end{equation*}

 Moreover, if $\nu(\dd \theta\, \dd \gamma) = \rho(\dd \gamma) G_0(\dd \theta) $, where $G_0$ is a probability measure on $\X$, then 
\begin{equation*}
    j_k(x_1,\ldots,x_k) = k! \, \prob( \Phi(\X) = k) \,\E\left[\prod_{h=1}^{|\underline{T}|} \int_{\X}\prod_{j: T_j = h} \kappa(x_j ; \theta) G_0(\dd \theta)\right],
\end{equation*}
where the expectation is taken with respect to the indicators $\underline{T} = (T_1, \ldots, T_k) \in \calT_k$ with distribution
\begin{equation*}
\prob(\underline{T} = \underline{t}) \propto  \prod_{h=1}^{|\underline{t}|} \int_{\R_+} e^{-\gamma} \gamma^{n_h} \rho(\dd \gamma) .    
\end{equation*}
\end{theorem}
\begin{proof}
The proof follows from specializing \Cref{thm:janossy_general} for a \textsc{sncp}. In particular, \Cref{thm:janossy_general} states that
\begin{equation}\label{eq:janossy_two_terms}
    J_k(\dd x_1\,\ldots\, \dd x_k) = \prob(\Phi^!_{(x_1,\ldots,x_k)}(\X) = 0)  M^{(k)}_{\Phi}(\dd x_1\,\ldots\, \dd x_k).
\end{equation}

Focusing on the first term of \eqref{eq:janossy_two_terms}, and denoting with $\underline{x} = (x_1,\ldots,x_k)$,
\begin{equation}\label{eq:ev_janossy_start}
        \prob \left( \Phi^!_{\underline{x}}(\X) = 0 \right) = \E \left[ \prob \left( \Phi^!_{\underline{x}}(\X) = 0 \mid \underline{T} \right) \right] = \E \left[ \prob \left( \Phi(\X) = 0 \right) \prod_{h=1}^{|\underline{T}|} \prob \left( \Phi_{\zeta_{\underline{x}_h}}(\X) = 0 \right)  \right], 
\end{equation}
where $\underline{T}$ and the $\Phi_{\zeta_{\underline{x}_h}}$'s are introduced in \Cref{thm:red_palm_sncp}. It follows that \eqref{eq:ev_janossy_start} writes as
\begin{equation*}
   \begin{aligned}
       \prob \left( \Phi^!_{\underline{x}}(\X) = 0 \right) &= \E\left[ \prob \left( \Phi(\X) = 0 \mid \Lambda \right)\right] \E \left[ \prod_{h=1}^{|\underline{T}|} \E\left[ \prob \left( \Phi_{\zeta_{\underline{x}_h}}(\X) = 0 \mid \zeta_{\underline{x}_h}\right)\right]  \right]\\
       &= \E\left[ e^{- \int_{\X} \int_{\X\times\R_+} \gamma \kappa(x; \theta) \Lambda(\dd \theta\, \dd \gamma) \dd x} \right] \E\left[ \prod_{h=1}^{|\underline{T}|} \E\left[ e^{- \int_{\X} \gamma_{\underline{x}_h} \kappa(x; \theta_{\underline{x}_h}) \dd x} \right] \right]\\
       &= \exp\left\{ - \int_{\X\times\R_+} \left(1 - e^{- \gamma } \right) \nu(\dd \theta \, \dd \gamma) \right\} \E\left[ \prod_{h=1}^{|\underline{T}|} \int_{\X\times\R_+} e^{-\gamma } f_{\underline{x}_h}(\dd\theta \, \dd \gamma) \right]\\
       &= \exp\left\{ - \int_{\X\times\R_+} \left(1 - e^{- \gamma  } \right) \nu(\dd \theta\, \dd \gamma) \right\} \sum_{\underline{t} \in \calT_k} \left[ \prod_{h=1}^{|\underline{t}|} \int_{\X\times\R_+} e^{-\gamma } f_{\underline{x}_h}(\dd\theta \, \dd \gamma) \right] \prob(\underline{T} = \underline{t}),
   \end{aligned} 
\end{equation*}
where $f_{\underline{x}_h}$ is defined in \Cref{thm:red_palm_sncp}.  Finally, plugging this last expression in \eqref{eq:janossy_two_terms} and using \Cref{lemma:k_factorial_sncp} for $M^{(k)}_{\Phi}$, we obtain
\begin{equation*}
    \begin{aligned}
        J_k(\dd x_1\,\ldots\,\dd x_k) &= \exp\left\{ - \int_{\X\times\R_+} \left(1 - e^{- \gamma } \right) \nu(\dd \theta \,\dd \gamma) \right\} \sum_{\underline{t} \in \calT_k} \prod_{h=1}^{|\underline{t}|} \eta(\underline{x}_h) \int_{\X\times\R_+} e^{-\gamma } f_{\underline{x}_h}(\dd\theta \, \dd \gamma)\; \dd x_1\,\ldots\,\dd x_k\\
        &= \exp\left\{ - \int_{\X\times\R_+} \left(1 - e^{- \gamma} \right) \nu(\dd \theta\, \dd \gamma) \right\} \sum_{\underline{t} \in \calT_k} \prod_{h=1}^{|\underline{t}|}  \int_{\X\times\R_+} e^{-\gamma } \gamma^{n_h} \prod_{j: t_j = h} \kappa(x_j; \theta) \nu(\dd\theta \, \dd \gamma) \; \dd x_1\,\ldots\,\dd x_k,
    \end{aligned}
\end{equation*}
and the thesis is proved.\\

If $\nu(\dd \theta\, \dd \gamma) = \rho(\dd \gamma) G_0(\dd \theta)$, where $G_0$ is a probability measure, then the Janossy density boils down to
\begin{equation}\label{eq:janossy_result_specialized}
\begin{split}
 & j_k(x_1,\ldots,x_k) \\
 \quad & = \exp\left\{ - \int_{\R_+} ( 1 - e^{-\gamma}) \rho(\dd \gamma) \right\} \sum_{\underline{t} \in \calT_k} \prod_{h=1}^{|\underline{t}|} \int_{\X} \prod_{j: t_j = h} \kappa(x_j ; \theta) G_0(\dd \theta) \int_{\R_+} e^{-\gamma} \gamma^{n_h} \rho(\dd \gamma).  
\end{split}
\end{equation}
Specializing the result in \Cref{prop:prior_partition_sncp} for $\nu(\dd \theta\, \dd \gamma) = \rho(\dd \gamma) G_0(\dd \theta) $, we obtain
\[
\prob(N, \{\calC_1,\ldots, \calC_C\}) = \frac{1}{N!} e^{- \int_{\R_+}(1-e^{-\gamma}) \rho(\dd \gamma)} \prod_{h = 1}^C \int_{\R_+} e^{-\gamma} \gamma^{n_h} \rho(\dd \gamma), 
\]
and the following holds true:
\begin{equation*}
\begin{aligned}
    \prob(N = k) &= \sum_{\{\calC_1,\ldots,\calC_C\}} \prob(N = k, \{\calC_1,\ldots, \calC_C\})\\
    &= \frac{1}{k!} e^{- \int_{\R_+}(1-e^{-\gamma}) \rho(\dd \gamma)} \sum_{\{\calC_1,\ldots,\calC_C\}} \prod_{h = 1}^C \int_{\R_+} e^{-\gamma} \gamma^{n_h} \rho(\dd \gamma)\\
    &= \frac{1}{k!} e^{- \int_{\R_+}(1-e^{-\gamma}) \rho(\dd \gamma)} \sum_{\underline{t} \in \calT_k} \prod_{h = 1}^{|\underline{t}|} \int_{\R_+} e^{-\gamma} \gamma^{n_h} \rho(\dd \gamma),
\end{aligned}
\end{equation*}
where the last equality is just due to a change in notation. Therefore, \eqref{eq:janossy_result_specialized} can be written as
\begin{equation*}
\begin{aligned}
 &  j_k(x_1,\ldots,x_k) =  k!\prob(N=k) \\
    &\quad \times \sum_{\underline{t} \in \calT_k} \left[ \prod_{h=1}^{|\underline{t}|} \int_{\X} \prod_{j: t_j = h} \kappa(x_i ; \theta) G_0(\dd \theta)\right] \frac{\exp\left\{ - \int_{\R_+} ( 1 - e^{-\gamma}) \rho(\dd \gamma) \right\} \prod_{h=1}^{|\underline{t}|}  \int_{\R_+} e^{-\gamma} \gamma^{n_h} \rho(\dd \gamma)}{k!\prob(N=k)},  
\end{aligned}
\end{equation*}
where the fraction term corresponds to $\prob(\underline{T}= \underline{t})$, where the law of $\underline{T}$ is defined in the statement. In conclusion, the thesis follows by expressing the last equation as an expected value with respect to $\underline{T}$.
\end{proof}

\subsection{Proofs of the auxiliary results in Appendix \ref{app:auxiliary_results}}\label{app:auxiliary_proofs}

\subsection*{Proof of \Cref{lemma:laplace_sncp}}

The proof is straightforwardly obtained by exploiting the decomposition in \eqref{eq:sncp_definition}:
\begin{equation*}
    \begin{aligned}
        \mathcal{L}_{\Phi}(f) &= \E\left[ \exp\left\{ - \int_{\X} f(x) \Phi(\dd x) \right\} \right] = \E\left[ \E \left[ \exp\left\{ - \int_{\X} f(x) \Phi(\dd x) \right\} \mid \Lambda \right] \right]\\
        &= \E\left[ \exp\left\{ - \int_{\X} (1 - e^{-f(x)}) \sum_{j\geq 1} \gamma_j \kappa(x; \theta_j) \dd x \right\}\right]\\
        &= \E\left[ \exp\left\{ - \int_{\X\times\R_+} \gamma \int_{\X} (1 - e^{-f(x)}) \kappa(x; \theta) \dd x \, \Lambda(\dd \theta\, \dd \gamma) \right\}\right]\\
        &= \exp\left\{ - \int_{\X\times\R_+} \left( 1 - \exp\left\{ - \gamma \int_{\X} (1 - \exp\{-f(x)\}) \kappa(x; \theta) \dd x \right\} \right) \nu(\dd \theta\, \dd \gamma) \right\}.
    \end{aligned}
\end{equation*}
The third equality follows from the fact that $\Phi\mid \Lambda$ is a Poisson process, while the last step holds because $\Lambda$ is Poisson process.

\subsection*{Proof of \Cref{lemma:laplace_cox_terms}}

\emph{Point (i)}. For the Laplace functional of $\Phi_{\zeta_{\underline{x}}}$, with $\underline{x} = (x_1,\ldots,x_k)$, the proof follows from direct computation. Let $f:\X \to \R_+$ be any measurable function, then
\begin{equation*}
    \begin{aligned}
        \mathcal{L}_{\Phi_{\zeta_{\underline{x}}}}(f) &= \E\left[ \exp\left\{ - \int_{\X} f(x) \Phi_{\zeta_{\underline{x}}}(\dd x) \right\} \right] = \E\left[ \E \left[ \exp\left\{ - \int_{\X} f(x) \Phi_{\zeta_{\underline{x}}}(\dd x) \right\} \mid \zeta_{\underline{x}} \right] \right]\\
        &= \E\left[ \exp\left\{ - \int_{\X} (1 - e^{-f(x)}) \gamma_{\underline{x}} \kappa(x; \theta_{\underline{x}}) \dd x \right\}\right]\\
        &= \int_{\X\times\R_+} \exp\left\{ -\gamma \int_{\X} (1 - \exp\{- f(x)\}) \kappa(x;\theta) \dd x \right\} \gamma^k \prod_{j=1}^k \kappa(x_j; \theta)\nu(\dd \theta\, \dd \gamma) / \eta(\underline{x}).
    \end{aligned}
\end{equation*}
The third equality follows from the fact that $\Phi_{\zeta_{\underline{x}}}\mid \zeta_{\underline{x}}$ is distributed as Poisson process.\\

\emph{Point (ii)}. Still by direct computation, the mean measure of $\Phi_{\zeta_{\underline{x}}}$ writes as
\begin{equation*}
    \begin{aligned}
        M_{\Phi_{\zeta_{\underline{x}}}}(\dd y) &= \E\left[ \E\left[  \Phi_{\zeta_{\underline{x}}}(\dd y) \mid \zeta_{\underline{x}} \right] \right]
        = \E\left[ \gamma_{\underline{x}} \kappa(y; \theta_{\underline{x}}) \dd y \right] = \\
        &= \int_{\X\times \R_+} \gamma \kappa(y; \theta) \dd y \; \gamma^{k} \prod_{j=1}^k \kappa(x_j; \theta) \nu(\dd \theta \, \dd \gamma) / \eta(\underline{x})\\
        &=   \int_{\X\times \R_+} \gamma^{k+1} \kappa(y; \theta)  \; \prod_{j=1}^k \kappa(x_j; \theta) \nu(\dd \theta \, \dd \gamma) \dd y / \eta(\underline{x})\\
        &= \eta(\underline{x}, y) \dd y / \eta(\underline{x})
    \end{aligned}
\end{equation*}
and the thesis is proved.\\

\emph{Point (iii)}. By \cite[Proposition 3.2.5]{BaBlaKa}, if $\Psi$ is a Cox process directed by the random measure $\mu$, then the reduced Palm version $\Psi^!_y$ is a Cox process directed by $\mu_y$ (the Palm version of $\mu$ at $y$). Now, observe that $\Phi_{\zeta_{\underline{x}}}$ is a Cox process directed by the random measure $\mu(\dd x ) = \gamma_{{\underline{x}}} \kappa(x; \theta_{{\underline{x}}}) \,\dd x $, with $ (\theta_{{\underline{x}}}, \gamma_{{\underline{x}}}) \sim f_{{\underline{x}}}(\dd \theta \,\dd \gamma)$ and $f_{{\underline{x}}}$ is defined in the statement of the present lemma. Consequently, the reduced Palm version $ \left( \Phi_{\zeta_{{\underline{x}}}} \right)^!_{y}$ is a Cox process directed by $\mu_y$. We are then left with determining the distribution of $\mu_y$, which we derive by computing its Laplace functional $\mathcal{L}_{\mu_y}(f)$, for any measurable function $f: \X \to \R_+$. To this end, rely on \cite[Proposition 3.2.1]{BaBlaKa}, which states that, for any measurable functions $f, g: \X \to \R_+$, the Palm distribution of a random measure $\mu$ satisfies
    \begin{equation}\label{eq:bacelli_palm_laplace}
        \evalat{\frac{\partial}{\partial t} \mathcal{L}_{\mu}(f + tg)}{t=0} =  - \E\left[ \mu(g) e^{-\mu(f)} \right] = - \int_{\X} g(y) \mathcal{L}_{\mu_y}(f) M_{\mu}(\dd y),
    \end{equation}
where $\mu(f):= \int_{\X} f(x) \mu(\dd x)$. Therefore, since $\mu(\dd x ) = \gamma_{{\underline{x}}} \kappa(x; \theta_{{\underline{x}}}) \,\dd x $, the term $\E\left[ \mu(g) e^{-\mu(f)} \right]$ writes as
    \begin{align}
        \E\left[ \mu(g) e^{-\mu(f)} \right] &= \E\left[ \int_{\X} g(y) \gamma_{\underline{x}} \kappa(y; \theta_{\underline{x}}) \dd y  \cdot \exp\left\{ - \int_{\X} f(z) \gamma_{\underline{x}} \kappa(z; \theta_{\underline{x}}) \dd z  \right\} \right]\notag \\
        &=  \int_{\X\times \R_+} \int_{\X} g(y) \gamma \kappa(y; \theta) \dd y  \cdot \exp\left\{ - \int_{\X} f(z) \gamma \kappa(z; \theta) \dd z  \right\} \cdot f_{\underline{x}}(\dd\theta \, \dd \gamma)\notag \\
        &= \int_{\X} g(y)  \int_{\X\times \R_+}  \exp\left\{ - \int_{\X} f(z) \gamma \kappa(z; \theta) \dd z  \right\} \gamma \kappa(y; \theta)  f_{\underline{x}}(\dd\theta \, \dd \gamma) \cdot  \dd y.\label{eq:mu_comp_derivative}
    \end{align}
Observing that $M_\mu(\dd y) = \E[\gamma_{\underline{x}} \kappa(y; \theta_{\underline{x}}) \dd y ] = \int_{\X\times \R_+} \gamma \kappa(y; \theta) f_{\underline{x}}(\dd\theta\, \dd \gamma)\, \dd y$,  \eqref{eq:mu_comp_derivative} can be expressed as
\begin{equation*}
    \begin{aligned}
         \E\left[ \mu(g) e^{-\mu(f)} \right] = \int_{\X} g(y)  \int_{\X\times \R_+}  \exp\left\{ - \int_{\X} f(z) \gamma \kappa(z; \theta) \dd z  \right\} \frac{\gamma \kappa(y; \theta)  f_{\underline{x}}(\dd\theta \, \dd \gamma)}{\int_{\X\times \R_+} \gamma^\prime \kappa(y; \theta^\prime) f_{\underline{x}}(\dd\theta^\prime\, \dd \gamma^\prime)} \cdot  M_\mu(\dd y)
    \end{aligned}
\end{equation*}
By identification in \eqref{eq:bacelli_palm_laplace}, we see that
\begin{equation*}
        \mathcal{L}_{\mu_y}(f) = \int_{\X\times \R_+}  \exp\left\{ - \int_{\X} f(z) \gamma \kappa(z; \theta) \dd z  \right\} \frac{\gamma \kappa(y; \theta)  f_{\underline{x}}(\dd\theta \, \dd \gamma)}{\int_{\X\times \R_+} \gamma^\prime \kappa(y; \theta^\prime) f_{\underline{x}}(\dd\theta^\prime\, \dd \gamma^\prime)},
\end{equation*}
which corresponds to say that $\mu_y(\dd x) = \gamma^* \kappa(x; \theta^*) \dd x$, with 
\[
( \theta^*, \gamma^*) \sim f_*(\dd \theta\,\dd \gamma) \propto \gamma \kappa(y; \theta)  f_{\underline{x}}(\dd\theta \, \dd \gamma) \propto f_{(\underline{x}, y)}(\dd\theta \, \dd \gamma).
\]
Summing up, we have found that 
\[
\left(\Phi_{\zeta_{{\underline{x}}}}\right)^!_y \mid (\theta^*, \gamma^*) \sim \textsc{pp}(\gamma^* \kappa(x; \theta^*) \,\dd x), \qquad (\theta^*, \gamma^*) \sim f_{(\underline{x}, y)}(\dd\theta \, \dd \gamma),
\]
which is equivalent to say $\left(\Phi_{\zeta_{{\underline{x}}}}\right)^!_y \deq \Phi_{\zeta_{({\underline{x}}, y)}}$, and the thesis is proved.

\subsection*{Proof of \Cref{lemma:k_factorial_sncp}}

For any $B_1,\ldots,B_k \in \Xcr$, it holds that
\begin{align}
        M^{(k)}_{\Phi}(B_1,\ldots,B_k) &= \E\left[ \sum_{X_1,\ldots, X_k \in \Phi}^{\neq} \prod_{j=1}^k \indicator_{B_j}(X_j) \right] = \E\left[ \E\left[ \sum_{X_1,\ldots, X_k \in \Phi}^{\neq} \prod_{j=1}^k \indicator_{B_j}(X_j)\mid \Lambda \right] \right] \notag \\
        &= \E\left[ M^{(k)}_{\Phi\mid \Lambda}(B_1,\ldots,B_k) \right] = \E\left[ \prod_{j=1}^k M_{\Phi\mid \Lambda}(B_j) \right]\notag \\
        &= \E\left[ \prod_{j=1}^k \int_{B_j} \int_{\X\times\R_+} \gamma \kappa(x_j; \theta) \Lambda(\dd\theta\, \dd \gamma) \dd x_j \right] \notag \\
        &= \E\left[  \int_{(\X\times\R_+)^k} \prod_{j=1}^k \gamma_j \int_{B_j} \kappa(x_j; \theta_j) \dd x_j \Lambda^k(\dd\underline{\theta} \, \dd \underline{\gamma})  \right] \notag \\
        &= \int_{(\X\times\R_+)^k} \prod_{j=1}^k \gamma_j \int_{B_j} \kappa(x_j; \theta_j) \dd x_j M_{\Lambda}^k(\dd\underline{\theta}\, \dd \underline{\gamma}), \label{eq:k_moment_start}
\end{align}
where $\Lambda$ is the Poisson process appearing in \eqref{eq:sncp_definition} and the symbol $\not =$ over the summation in the first line means that the sum is extended over all pairwise distinct points of $\Phi$.

Moreover, from \cite[Lemma 14.E.4]{BaBlaKa}, the following holds true for any point process $\xi$ on a Polish space $\Y$ (embedded with the Borel $\sigma$-algebra $\Ycr$), for any $A_1,\ldots,A_k \in \Ycr$,
\begin{equation*}
    M_{\xi}^k(A_1,\ldots,A_k) = \sum_{q=1}^k \sum_{\{J_1,\ldots,J_q\}} M_{\xi}^{(q)}\left(\prod_{h=1}^q \left(\cap_{m \in J_h} A_m \right)\right),
\end{equation*}
where the summation is over all partitions $\{J_1,\ldots,J_q\}$ of $\{1,\ldots,k\}$. In the special case of $\xi$ being a Poisson process, using a limit argument, we obtain 
\begin{equation}\label{eq:k_moment_poisson_decomposition}
M_{\xi}^k(\dd y_1\,\ldots\,\dd y_k) = \sum_{q=1}^k \sum_{\{J_1,\ldots,J_q\}} \prod_{h=1}^q M_{\xi}(\dd y_{(J_{h})_1}) \prod_{m \in J_h} \delta_{y_{(J_{h})_1}}(\dd y_m),
\end{equation}
where $(J_{h})_1$ indicates the first index of $J_h$, according to any arbitrary ordering.

Then, since $\Lambda$ in \eqref{eq:k_moment_start} is a Poisson process, we can apply the expression in \eqref{eq:k_moment_poisson_decomposition}, leading to
\begin{equation*}
    \begin{aligned}
        M^{(k)}_{\Phi}(B_1,\ldots,B_k) &= \int_{(\X\times\R_+)^k} \left[ \prod_{j=1}^k \gamma_j \int_{B_j} \kappa(x_j; \theta_j) \dd x_j\right] \\
        &\qquad \times \sum_{q=1}^k \sum_{\{J_1,\ldots,J_q\}} \prod_{h=1}^q M_{\Lambda}(\dd \theta_{(J_{h})_1} \dd \gamma_{(J_{h})_1}) \prod_{m \in J_h} \delta_{(\theta_{(J_{h})_1}, \gamma_{(J_{h})_1})}(\dd \theta_m \, \dd \gamma_m)\\
        &= \sum_{q=1}^k \sum_{\{J_1,\ldots,J_q\}} \prod_{h=1}^q \int_{(\X\times\R_+)^{k_h} } \left[ \prod_{m \in J_h} \gamma_{m} \int_{B_m} \kappa(x_m; \theta_m) \dd x_m\right] \\
        &\qquad \times M_{\Lambda}(\dd \theta_{(J_{h})_1}\, \dd \gamma_{(J_{h})_1}) \prod_{m \in J_h} \delta_{(\theta_{(J_{h})_1}, \gamma_{(J_{h})_1})}(\dd \theta_m \,\dd \gamma_m)\\
        &= \sum_{q=1}^k \sum_{\{J_1,\ldots,J_q\}} \prod_{h=1}^q \int_{\X\times\R_+ }  \gamma^{k_h} \prod_{m \in J_h} \int_{B_{m}} \kappa(x_m; \theta) \dd x_m \; M_{\Lambda}(\dd \theta\, \dd \gamma)\\
        &= \sum_{q=1}^k \sum_{\{J_1,\ldots,J_q\}} \prod_{h=1}^q \int_{\times_{m\in J_h} B_m} \int_{\X\times\R_+ }  \gamma^{k_h} \prod_{m \in J_h}  \kappa(x_m; \theta) M_{\Lambda}(\dd \theta\, \dd \gamma) \; \prod_{m \in J_h} \dd x_m,
    \end{aligned}
\end{equation*}
where $k_h$ is the cardinality of set $J_h$. For $\underline{x}_{J_h} = (x_m: m \in J_h)$, define 
\[
\eta(\underline{x}_{J_h}) = \int_{\X\times \R_+} \gamma^{k_h} \prod_{m \in J_h} \kappa(x_m; \theta) M_{\Lambda}(\dd\theta \, \dd \gamma).
\]
Therefore, the previous expression writes as
\begin{equation*}
\begin{aligned}
    M^{(k)}_{\Phi}(B_1,\ldots,B_k) &=  \sum_{q=1}^k \sum_{\{J_1,\ldots,J_q\}} \prod_{h=1}^q \int_{\times_{m\in J_h} B_m} \eta(\underline{x}_{J_h}) \; \prod_{m \in J_h} \dd x_m\\
    &= \sum_{q=1}^k \sum_{\{J_1,\ldots,J_q\}}  \int_{B_1 \times \cdots \times B_k} \prod_{h=1}^q \eta(\underline{x}_{J_h}) \; \dd x_1\,\ldots\, \dd x_{k}\\
    &= \int_{B_1 \times \cdots \times B_k} \sum_{q=1}^k \sum_{\{J_1,\ldots,J_q\}}   \prod_{h=1}^q \eta(\underline{x}_{J_h}) \; \dd x_1\,\ldots\, \dd x_{k},
\end{aligned}   
\end{equation*}
and the thesis follows, with the introduction of the allocation variables $\underline{t}\in \calT_k$.

\subsection*{Proof of \Cref{prop:prior_partition_sncp}}

The formulation of the \textsc{sncp} as cluster process in \eqref{eq:sncp_definition} yields a clustering structure at the latent level. Based on this, the generating mechanism of the points can be described in terms of the clustering structure and the cluster-specific parameters of the kernel $\kappa$. Specifically, consider the event $((\theta^*_1,N^*_1),\ldots,(\theta^*_C,N^*_C))$, where $C$ denotes the number of clusters among the $N$ points of $\Phi$, $\theta^*_h$ represents the parameter of kernel $\kappa$ in cluster $h$, and $N^*_h$ is the number of points of $\Phi$ in cluster $h$, given a uniform random ordering of the clusters. It holds that $N= \sum_{h=1}^C N^*_h$ and the cluster-specific parameters $\theta^*_h$'s are distinct. Denote with $\underline{\theta}^* = (\theta^*_1,\ldots,\theta^*_C)$. The main computation concerns the probability of the event just described, that is
\begin{equation*}
\begin{aligned}
    \prob((\theta^*_1,N^*_1),\ldots,(\theta^*_C,N^*_C)) &= \E\left[ \prob((\theta^*_1,N^*_1),\ldots,(\theta^*_C,N^*_C) \mid \Lambda) \right]\\
    &= \frac{1}{C!} \E\left[ \int_{(\X\times\R_+)^C} \prod_{h=1}^C \frac{1}{{N^*_h}!} e^{-\gamma_h}  \gamma_h^{N^*_h} \delta_{\theta^*_h}(\theta_h) \cdot \prod_{j > C} e^{-\gamma_j} \Lambda^C (\dd \underline{\theta} \, \dd \underline{\gamma}) \right]\\
     &= \frac{1}{C!}\E\left[ \int_{(\X\times\R_+)^C}  e^{-\int_{\X\times\R_+} s \Lambda(\dd x\, \dd s)} \cdot \prod_{h=1}^C \frac{1}{{N^*_h}!}  \gamma_h^{N^*_h} \delta_{\theta^*_h}(\theta_h)  \Lambda^C (\dd\underline{\theta}\, \dd \underline{\gamma}) \right],
\end{aligned}
\end{equation*}
where $\Lambda^C$ is the $C$-th power of $\Lambda$.
Since the term $\prod_{h=1}^C \delta_{\theta^*_h}(\theta_h)$ entails that the integrand is zero on sets of the type
\[
    \{(\underline{\theta}, \underline{\gamma}) \in (\X\times\R_+)^C : \; \theta_i = \theta_j \text{ for } i \neq j \},
\]
then we can replace $\Lambda^C$ with the $C$-th factorial power $\Lambda^{(C)}$. By applying the \textsc{clm} formula, we obtain
\begin{align}
   & \prob((\theta^*_1,N^*_1),\ldots,(\theta^*_C,N^*_C)) \notag \\
    &\quad= \frac{1}{C!}\E\left[ \int_{(\X\times\R_+)^C}  e^{-\int_{\X\times\R_+} s \Lambda(\dd x\, \dd s)} \cdot \prod_{h=1}^C \frac{1}{{N^*_h}!}  \gamma_h^{N^*_h} \delta_{\theta^*_h}(\theta_h)   \Lambda^{(C)} (\dd\underline{\theta}\, \dd \underline{\gamma}) \right] \notag \\
    &\quad= \frac{1}{C!} \int_{(\X\times\R_+)^C} \E\left[ e^{-\int_{\X\times\R_+} s \Lambda^!_{\underline{\theta}, \underline{\gamma}}(\dd x \, \dd s)} \right] e^{-\sum_{h=1}^C \gamma_h} \prod_{h=1}^C \frac{1}{{N^*_h}!}  \gamma_h^{N^*_h} \delta_{\theta^*_h}(\theta_h) \;   M_\Lambda^{(C)} (\dd\underline{\theta}\, \dd \underline{\gamma}) \notag  \\
    &\quad= \frac{1}{C!} \int_{\R_+^C} \E\left[ e^{-\int_{\X\times\R_+} s \Lambda^!_{\underline{\theta}^*, \underline{\gamma}}(\dd x \, \dd s)} \right]  \prod_{h=1}^C \frac{1}{{N^*_h}!}  e^{-\gamma_h} \gamma_h^{N^*_h}   \; M_\Lambda^{(C)} (\dd\underline{\theta}^* \,\dd \underline{\gamma}). \label{eq:partition_after_clm}
\end{align}
Since $\Lambda$ is a Poisson process with intensity measure $\nu(\dd\theta\, \dd\gamma)$, then \eqref{eq:partition_after_clm} yields
\begin{equation*}
\begin{aligned}
    \prob((\theta^*_1,N^*_1),\ldots,(\theta^*_C,N^*_C))&= \frac{1}{C!}\E\left[ e^{-\int_{\X\times\R_+} s \Lambda(\dd x\, \dd s)} \right]  \int_{\R_+^C}  \prod_{h=1}^C \frac{1}{{N^*_h}!}  e^{-\gamma_h} \gamma_h^{N^*_h}   \; \nu^C (\dd\underline{\theta}^*\, \dd \underline{\gamma})\\
    &= \frac{1}{C!} e^{-\int_{\X\times\R_+} (1 - e^{-\gamma}) \nu(\dd\theta \, \dd\gamma)} \prod_{h=1}^C \int_{\R_+}   \frac{1}{{N^*_h}!}  e^{-\gamma_h} \gamma_h^{N^*_h}   \; \nu (\dd \theta_h^* \, \dd \gamma_h).
\end{aligned}
\end{equation*}
For any clustering $(\calC_1,\ldots, \calC_C)$ of the $N$ points of $\Phi$, with cluster cardinalities $N^*_h = |\calC_h|$, $h=1,\ldots,C$, note that the following relationship holds
\[
\prob((\theta^*_1,\calC_1),\ldots, (\theta^*_C,\calC_C)) = \binom{N}{N^*_1,\ldots,N^*_k}^{-1} \prob((\theta^*_1,N^*_1),\ldots,(\theta^*_C,N^*_C)).
\]
The marginal distribution of $(\calC_1,\ldots,\calC_C)$ is simply obtained by integrating out $(\theta^*_1,\ldots,\theta^*_C)$ from the previous expression, leading to
\begin{equation*}
    \prob(\calC_1,\ldots,\calC_C) =\frac{1}{N!} \frac{1}{C!} e^{-\int_{\X\times\R_+} (1 - e^{-\gamma}) \nu(\dd\theta \, \dd\gamma)} \prod_{h=1}^C \int_{\X\times\R_+}    e^{-\gamma_h} \gamma_h^{N^*_h}   \; \nu (\dd \theta_h \, \dd \gamma_h). 
\end{equation*}
Observe that $\prob(\{\calC_1,\ldots, \calC_C\})$ is obtained by multiplying by $C!$ the previous expression. In the statement of the theorem, we stress that this probability distribution describes the number of points $N$ as well.

\end{document}